\documentclass{amsart}

\usepackage{amssymb,amsfonts,amsmath,nccmath,amsthm}
\usepackage[all,arc]{xy}
\usepackage{tikz-cd}
\usepackage{tikz}
\usetikzlibrary{arrows,decorations}

\usepackage
[pdfauthor={Catherine Hsu, Alice Pozzi, Preston Wake, and Carl Wang-Erickson},
pdftitle={Eisenstein congruences in tame families and the class group of a metabelian extension},
bookmarks=false]
{hyperref}

\hypersetup{
  colorlinks   = true, 
  urlcolor     = blue, 
  linkcolor    = blue, 
  citecolor   = blue 
}

\usepackage{enumitem}
\usepackage{mathrsfs}
\usepackage{multicol}
\usepackage{xcolor,colortbl}
\usepackage{graphics}
\usepackage{color}

\setlist[enumerate]{leftmargin=2.5em}
\setlist[itemize]{leftmargin=2.5em}

\newcommand\define{\newcommand}

\define\rk{\mathrm{rk}}

\define\isoto{\xrightarrow{\sim}}

\newcommand{\dia}[1]{{\langle #1 \rangle}}

\newcommand{\ttmat}[4]{\left( \begin{array}{cc}
#1 & #2 \\
#3 & #4
\end{array}
\right)}

\newcommand{\Z}{\mathbb{Z}}
\newcommand{\Q}{\mathbb{Q}}

\newcommand{\F}{\mathbb{F}}

\newcommand{\m}{\mathfrak{m}}

\newcommand{\varep}{\varepsilon}

\newcommand{\Gal}{\mathrm{Gal}}

\newcommand{\Fr}{\mathrm{Fr}}

\newcommand{\red}{\mathrm{red}}

\define\Cl{{\mathrm{Cl}}}
\define\GL{{\mathrm{GL}}}

\define\SL{{\mathrm{SL}}}

\define{\Fitt}{\mathrm{Fitt}}
\define{\Ann}{\mathrm{Ann}}

\newtheorem{thm}{Theorem}

\newtheorem{prop}[thm]{Proposition}
\newtheorem{lem}[thm]{Lemma}

\newtheorem{conj}[thm]{Conjecture}

\theoremstyle{definition}
\newtheorem{defn}[thm]{Definition}

\newtheorem{eg}[thm]{Example}

\theoremstyle{remark}
\newtheorem{rem}[thm]{Remark}

\newcommand{\lra}{\longrightarrow}

\newcommand{\rinj}{\hookrightarrow}

\newcommand{\rsurj}{\twoheadrightarrow}

\DeclareMathOperator{\ad}{ad}

\newcommand{\bT}{\mathbb{T}}

\newcommand{\Db}{{\bar D}}

\newcommand{\Tr}{\mathrm{Tr}}

\newcommand{\sm}[4]{\ensuremath{\big(\begin{smallmatrix}#1 & #2 \\ #3 & #4\end{smallmatrix}\big)}}

\newcommand{\up}[1]{^{(#1)}}

\newcommand{\oQ}{\overline{\Q}}
\newcommand{\cS}{\mathcal{S}}

\newcommand{\ep}{\epsilon}

\newcommand{\ov}[1]{\overline{#1}}

\numberwithin{thm}{section}

\author{Catherine Hsu}
\address{Department of Mathematics and Statistics, Swarthmore College, Swarthmore, PA, 19081}
\email{chsu2@swarthmore.edu}

\author{Alice Pozzi}
\address{School of Mathematics, University of Bristol, Bristol, United Kingdom BS8 1TW}
\email{alice.pozzi@bristol.ac.uk}

\author{Preston Wake}
\address{Department of Mathematics, Michigan State University, East Lansing, MI 48824,USA}
\email{wakepres@msu.edu}

\author{Carl Wang-Erickson}
\address{Department of Mathematics, University of Pittsburgh, Pittsburgh, PA 15260, USA}
\email{carl.wang-erickson@pitt.edu}

\title[Eisenstein congruences and class groups]{Eisenstein congruences in tame families and the class group of a metabelian extension}

\begin{document}

\begin{abstract}
Ribet's 1976 proof of the converse to Herbrand's theorem created a new paradigm in algebraic number theory by illustrating that unramified abelian Galois extensions of number fields can be constructed using Galois representations associated to Eisenstein congruences, i.e., cuspidal eigenforms that are congruent to Eisenstein series. In recent years, further progress has been made in investigating the quantity of such Eisenstein congruences, showing that when there are many such congruences, additional and finer information can be gleaned about the arithmetic structure of the related extensions. In this vein, we use new results on the existence of Eisenstein congruences in tame families to explain a computational observation about splitting behavior of primes in a certain metabelian number field. 
\end{abstract}

\maketitle

\tableofcontents

\section{Introduction}
\label{sec: intro}

Eisenstein congruences played an integral role in both Ribet's converse to Herbrand's theorem \cite{ribet1976} as well as Mazur's celebrated torsion theorem \cite{mazur1978}: in the former, Ribet utilized Galois representations associated to Eisenstein congruences that arise from a prime dividing an $L$-value, and in the latter, Mazur studied Eisenstein congruences that arise from a prime dividing an Euler factor.
Ribet's technique for producing Galois cocyles using Eisenstein congruences has been fruitfully reinterpreted using the language of Galois deformation theory \cite{mazur1989}, notably by Bella\"iche and Chenevier \cite{BC2009}. Calegari and Emerton \cite{CE2005} applied Galois deformation theory to Mazur's original setting to give a partial answer to a question of Mazur about the number of Eisenstein congruences, showing that this number has deep arithmetic significance related to class numbers of certain number fields. A series of works by various subsets of the present authors have pushed this direction further, giving a complete answer to Mazur's question \cite{WWE3}, generalizing some of Mazur's results to certain squarefree levels \cite{WWE5}, analyzing a particularly pathological squarefree case in depth \cite{part1,part2}, and generalizing these results to include modular forms with nontrivial character \cite{HPWE}. In this paper, we use the abundance of Eisenstein congruences established in \cite{HPWE} together with techniques of \cite{part1} to explain a class-field-theoretic phenomenon that was observed computationally in \cite{part2}.

\subsection{Class fields and the rank of Mazur's Eisenstein ideal}
\label{subsec: intro Mazur}
Let $p, \ell_0 \ge 5$ be primes with $\ell_0 \equiv 1 \pmod{p}$. For a squarefree integer $n$, let $\bT_{\Gamma_0(n)}^\varep$ denote the $\Z_p$-Hecke algebra acting on weight-two modular forms of level $\Gamma_0(n)$ localized at the Eisenstein maximal ideal containing $p$ and $(\mathbf{w}_\ell+1)$, for each prime $\ell \mid n$, where $\mathbf{w}_\ell$  denotes the Atkin--Lehner involution at $\ell$. Also, let $\bT_{\Gamma_0(\ell_0)}^{\circ,\varep}$ be the maximal quotient acting faithfully on cusp forms.\footnote{This notation matches \cite{WWE5} in which $\varep$ denotes a choice of Atkin--Lehner signature; we primarily consider $\varep=(-1,-1)$ in this paper.} Mazur's theorem \cite{mazur1978} implies $\bT_{\Gamma_0(\ell_0)}^{\circ,\varep}$ is nonzero and free of finite rank as a $\Z_p$-module, and he asked if there is any arithmetic significance to its rank. Calegari and Emerton \cite{CE2005} proved that, if $\rk_{\Z_p}\bT_{\Gamma_0(\ell_0)}^{\circ,\varep} >1$, then $\dim_{\F_p} \Cl(\Q(\sqrt[p]{\ell_0}))[p] >1$. This dimension is at least one by genus theory:
letting $\zeta_{\ell_0}\up p$ denote a primitive element for the unique degree $p$ subextension of $\Q(\zeta_{\ell_0})/\Q$, one can show that $\Q(\sqrt[p]{\ell_0},\zeta_{\ell_0}\up p)/\Q(\sqrt[p]{\ell_0})$ is unramified. So, the condition on the class group can be interpreted as asking 
if there is any other unramified $p$-extension.

One of the main results of \cite{WWE3} can be interpreted as a more precise relation between the rank $\bT_{\Gamma_0(\ell_0)}^{\circ,\varep}$ and another class group.   Let $L/\Q(\zeta_p)$ be the unique Galois $p$-extension on which $\Gal(\Q(\zeta_p)/\Q)$ acts by the inverse $\omega^{-1}$ of the mod-$p$ cyclotomic character $\omega$, and in which $1-\zeta_p$ splits completely and only primes above $\ell_0$ ramify (this extension exists and is unique by \cite[Lem.\ 3.9]{CE2005}). Then, $L(\zeta_{\ell_0}\up p)/\Q(\zeta_{\ell_0}\up p,\zeta_p)$ is an unramified degree-$p$ Galois extension, and hence defines a nonzero element of $\Cl(\Q(\zeta_{\ell_0}\up p,\zeta_p))[p]_{\omega^{-1}}$. In \cite{WWE3}, it is proven that $\rk_{\Z_p}\bT_{\Gamma_0(\ell_0)}^{\circ,\varep} >1$ if and only if 
$\dim_{\F_p}\Cl(\Q(\zeta_{\ell_0}\up p,\zeta_p))[p]_{\omega^{-1}} >1.$

Both \cite{CE2005} and \cite{WWE3} work with Galois deformation rings as an intermediary. That is, a deformation ring 
is defined and proven to be isomorphic to $\bT_{\Gamma_0(\ell_0)}^\varep$. The two papers take different approaches to the definition of such a deformation ring, but the resulting analysis is similar in spirit; in the discussion below, we adopt the choice of $R_{\Gamma_0(\ell_0)}^\varep$ as in \cite{WWE3}. In accordance with Mazur's description of the tangent space of deformation rings in \cite{mazur1989}, homomorphisms $R_{\Gamma_0(\ell_0)}^\varep \to \F_p[\ep]/(\ep^2)$, referred to as tangent vectors, are related to $1$-cocycles in Galois cohomology. In both \cite{CE2005} and \cite{WWE3}, it is shown that the tangent space of $R_{\Gamma_0(\ell_0)}^\varep$ is one-dimensional (this is related to Mazur's theorem on the local principality of the Eisenstein ideal), and hence $\rk_{\Z_p}\bT_{\Gamma_0(\ell_0)}^{\circ,\varep} >1$ is equivalent to the existence of deformation of this tangent vector to homomorphism to $R_{\Gamma_0(\ell_0)}^\varep \rsurj \F_p[\ep]/(\ep^3)$. In \cite{CE2005}, explicit deformation theory is used to construct a new unramified $p$-extension of $\Q(\sqrt[p]{\ell_0})$ using such a deformation. In \cite{WWE3}, a strategy is developed wherein, if $R$ is a deformation ring, $\phi: R \rsurj \F_p[\ep]/(\ep^2)$ is a tangent vector, and $S$ is an Artinian local $\F_p$-algebra, then it is shown that there are obstructions to deforming $\phi$ to a homomorphism $R \to S$, and that these obstructions are described by cup products (and, more generally, Massey products) in Galois cohomology related to $1$-cocycles determined by $\phi$. (This strategy further developed and articulated in \cite{CarlA-inf}.) In particular, the existence of a surjective homomorphism $R_{\Gamma_0(\ell_0)}^\varep \rsurj \F_p[\ep]/(\ep^3)$ implies the vanishing of certain cup products, and it is shown that these cup products are related to the class groups $\Cl(\Q(\sqrt[p]{\ell_0}))[p]$ and $\Cl(\Q(\zeta_{\ell_0}\up p,\zeta_p))[p]_{\omega^{-1}}$. 

In this paper, we exploit the existence of a large supply of Eisenstein congruences to deduce finer properties of certain class groups, using surjections $R \rsurj S$ and vanishing of Massey products as intermediate steps. 
To produce sufficiently many Eisenstein congruences, we  leverage the fact that such congruences deform in tame---that is, group-ring-valued---families. This input can be obtained by studying congruences either at level $\Gamma_1(n)$, as in \cite{HPWE}, or $\Gamma_0(n^2)$, as in \cite{LPW} for a squarefree $n$. The two settings are related by twisting, as explained in \cite{LMP}.

\subsection{Adding a bad prime}
\label{subsec: intro add bad prime}
Henceforth, we assume that $\rk_{\Z_p}\bT_{\Gamma_0(\ell_0)}^{\circ,\varep} =1$ and consider what happens when a second prime is added to the level. Let $\ell_1$ be a prime with $\ell_1 \not \equiv \pm 1 \pmod{p}$ such that $\ell_1$ is a $p$th power modulo $\ell_0$, which means that $\ell_1$ is a `bad prime' in the sense of Mazur \cite{mazur1978}, i.e., $T_{\ell_1}-(\ell_1+1)$ does not generate the Eisenstein ideal of $\bT_{\Gamma_0(\ell_0)}^{\circ,\varep}$. With $N=\ell_0\ell_1$, a result of Ribet, recorded by Yoo in \cite{yoo2019}, guarantees that there is a newform of level $\Gamma_0(N)$ that is congruent to the Eisenstein series of level $\Gamma_0(N)$ characterized by having Atkin--Lehner signature $\varep=(-1,-1)$. This case is particularly interesting from the point of view of \cite{WWE5}, in that the associated Hecke algebra $\bT_{\Gamma_0(N)}^\varep$ may not be Gorenstein (in fact, we expect that it never is---in every computed example, it is not). 

In \cite{WWE5}, for squarefree levels $n$, a deformation ring $R_{\Gamma_0(n)}^\varep$ and a surjective $\Z_p$-algebra homomorphism $R_{\Gamma_0(n)}^\varep\rsurj \bT_{\Gamma_0(n)}^\varep$ are constructed. For many levels $n$, Wiles's numerical criterion is used to show that $R_{\Gamma_0(n)}^\varep = \bT_{\Gamma_0(n)}^\varep$ and that both rings are Gorenstein; when $\bT_{\Gamma_0(N)}^\varep$ is not Gorenstein, this technique cannot work for $n=N$.

Nonetheless, in \cite{part1}, a new strategy for proving $R_{\Gamma_0(N)}^\varep =\bT_{\Gamma_0(N)}^\varep$ is developed using upper bounds on the rank of $R_{\Gamma_0(N)}^\varep$. The rank of $\bT_{\Gamma_0(N)}^\varep$ is known to be at least $3$ (accounting for the newform, the cusp form of level $\ell_0$, and the Eisenstein series). Therefore, if one can show the rank of $R_{\Gamma_0(N)}^\varep$ is at most $3$, then the surjective map $R_{\Gamma_0(N)}^\varep \rsurj \bT_{\Gamma_0(N)}^\varep$ must be an isomorphism. 

The main result of \cite{part1,part2} gives an upper bound on the rank of $R_{\Gamma_0(N)}^\varep$ in terms of class-field-theoretic conditions. Let $M=L(\sqrt[p]{\ell_1})$, and let $G=\Gal(M/\Q) \simeq (\F_p \times \F_p) \rtimes \F_p^\times$. For a finite set $S$ of primes of $M$, let $\Cl_S(M)$ be the $S$-class group of $M$ and let $\Cl_S(M)_{p,G}$ denote the largest quotient on which $p$ and $G$ act trivially. By class field theory, $\Cl_S(M)_{p,G}$ equals the Galois group of the maximal unramified elementary-abelian $p$-extension of $M$ that is Galois over $\Q$ and in which all primes in $S$ split completely. Let $T_0$ be the set of primes dividing $\ell_0$ and let $T$ be the set of primes dividing $N$, so that there is a natural surjective map $\Cl_{T_0}(M) \to \Cl_{T}(M)$. In \cite{part1,part2}, particularly \cite[Thm.\ 1.1.1]{part1} and \cite[Prop.\ 6.1.2]{part2}, it is shown that if $\dim_{\F_p} R_{\Gamma_0(N)}^\varep/pR_{\Gamma_0(N)}^\varep>3$, then 
\begin{equation}
\label{eq: a1 at ell1 CFT}
    \Cl_{T_0}(M)_{p,G} \to \Cl_{T}(M)_{p,G} \quad \text{ is injective}.
\end{equation}
In other words, if $M'/M$ denotes the maximal unramified elementary-abelian $p$-extension of $M$ that is Galois over $\Q$ and has all primes in $T_0$ split completely, then \eqref{eq: a1 at ell1 CFT} says that all primes in $T$ split completely in $M'$. This result is in line with the general strategy relating the size of deformation rings to the vanishing of Massey product because, as is shown in \cite{part2}, \eqref{eq: a1 at ell1 CFT} is equivalent to the vanishing of a certain triple Massey product in Galois cohomology.

In fact, in \cite{part1}, an invariant $\alpha^2+\beta \in \F_p$ is defined, and it is shown that  $\dim_{\F_p} R_{\Gamma_0(N)}^\varep/pR_{\Gamma_0(N)}^\varep >3$ if and only if both \eqref{eq: a1 at ell1 CFT} holds and $\alpha^2+\beta$ is zero. The definition of the invariant $\alpha^2+\beta$ is not important for this paper, except to say that it is derived using the ``unramified-or-Steinberg at $\ell_0$" property of Galois representations. In terms of modular forms, this means that it is intimately related to having level $\Gamma_0(N)$ (as opposed to some deeper level). On the other hand, \eqref{eq: a1 at ell1 CFT} is derived just from the existence of a certain Galois representation, without any particular constraint on the level. In \cite[\S5.3]{part2}, examples are computed showing that $\alpha^2+\beta$ is not always zero (and the computations suggest that it typically is not). However, \eqref{eq: a1 at ell1 CFT} is shown to hold in every example. The goal of this paper is to prove that \eqref{eq: a1 at ell1 CFT} always holds, assuming a conjecture of \cite{HPWE} about modular forms of deeper level. Since \eqref{eq: a1 at ell1 CFT} is known if $\rk_{\Z_p} \bT_{\Gamma_0(N)}^{\circ,\varep} >2$ by the results of \cite{part1,part2}, we assume for the remainder of the paper that 
\begin{equation}
\label{eq: min rk}
\rk_{\Z_p} \bT_{\Gamma_0(N)}^{\circ,\varep} =2.
\end{equation}

\subsection{Relaxing the condition at $\ell_0$}
\label{subsec: intro relax ell0}
In this paper, we work with a new deformation ring $R$. It is similar to $R_{\Gamma_0(N)}^\varep$, but the condition at $\ell_0$ is relaxed, so that it allows for Galois representations associated to cusp forms of level $\Gamma_0(N\ell_0)$ (or, by twisting, forms of level $\Gamma_1(N)$), and there is a surjective map $R \to R_{\Gamma_0(N)}^\varep$. Let $J^\red \subset R$ be the ideal such that $R/J^\red$ represents the deformations that remain reducible.
In \cite[\S4.3]{part1}, a distinguished tangent vector $R_{\Gamma_0(N)}^\varep \rsurj \F_p[\epsilon]/(\epsilon^2)$ is defined that is irreducible, in the sense that it does not send $J^\red$ to zero, and whose corresponding Galois representation cuts out $M'/\Q$; this defines an irreducible tangent vector $\phi: R \rsurj \F_p[\ep]/(\ep^2)$ of $R$ by precomposition.

The strategy is to show that, for an Artinian local ring $S$, the existence of a surjection $R \rsurj S$ deforming $\phi$ implies the vanishing of a certain triple Massey product. For our purposes, we define $S$ to be a ring of the form 
\begin{equation}
\label{eq: define S}
S = \frac{\F_p[y,z]}{(y^3,y^2-\mu z^2,yz)}
\end{equation}
for some $\mu \in \F_p^\times$
(we suppress dependence on $\mu$ from the notation); note that the $\F_p$-dimension of $S$ is $4$.
Let $\phi_S : S \to \F_p[\ep]/(\ep^2)$ be the tangent vector given by $\phi_S(y) =\ep$, $\phi_S(z) = 0$. Crucially,  as $y^2 \ne 0$ in $S$, the existence of a map $\pi: R \rsurj S$ such that $\phi_S \circ \pi =\phi$ is related to vanishing of triple Massey products of the cocycles used to define $\phi$ , as follows from the general strategy of \cite{WWE3,CarlA-inf}.

\begin{thm}
    \label{thm: main}
    Let $S$ be as in \eqref{eq: define S} and assume that $\rk_{\Z_p} \bT_{\Gamma_0(\ell_0)}^\circ=1$. Suppose that there is a surjection $\pi:R \rsurj S$ such that $\phi_S \circ \pi$ factors through $R_{\Gamma_0(N)}^\varep$ and $\pi(J^\red)=yS$. 
    Then the map $\Cl_{T_0}(M)_{p,G} \to \Cl_{T}(M)_{p,G}$ is injective. 
\end{thm}

See \S\ref{sec: main proof} for the proof of the theorem. An analogue of this theorem with $R$ replaced by $R_{\Gamma_0(N)}^\varep$ is proven in \cite{part1}, and the proof here is essentially the same, noting that the $\ell_0$-local condition defining $R_{\Gamma_0(N)}^\varep$ is not used in the proof. The advantage of Theorem \ref{thm: main} is that, while $R_{\Gamma_0(N)}^\varep/pR_{\Gamma_0(N)}^\varep$ often has $\F_p$-dimension $3$, and hence cannot surject onto $S$, the ring $R$ is much bigger and we can hope to find the map $\pi$ unconditionally. One reason why $R$ is  bigger than $R_{\Gamma_0(N)}^\varep$ is that Eisenstein congruences deform in tame families, so the Hecke algebra $\bT^\circ_{\Gamma_1(N)}$, acting on the Eisenstein localization of the cusp forms of level $\Gamma_1(N)$, is bigger than $\bT^{\circ,\varep}_{\Gamma_0(N)}$. (Note that when the superscript $\varep$ in  Hecke algebra notation is omitted, it means the Hecke algebra contains $U_\ell$ Hecke operators in place of the Atkin--Lehner involutions---see \S\ref{subsec: Hecke U vs w} for a precise definition of $\bT_{\Gamma_j(N)}^\circ$ $(j=0,1)$ and discussion of the choices of ramified Hecke operators.)

\subsection{Tame families and the structure of $\bT^\circ_{\Gamma_1(N)}$}
\label{subsec: intro tame family}
We now explain why $\bT^\circ_{\Gamma_1(N)}$ is large enough to admit a surjection to $S$, so that we expect the map $\pi$ of Theorem \ref{thm: main} can be constructed as a composite of $R \rsurj \bT_{\Gamma_1(N)}^\circ$ and $\bT_{\Gamma_1(N)}^\circ \rsurj S$.

Let $\Delta$ be the $p$-Sylow subgroup of $(\Z/N\Z)^\times$, and let $\Lambda=\Z_p[\Delta]$; fix a choice of generator of $\Delta$, which gives a generator $z \in \Lambda$ of its augmentation ideal.
Then, the Eisenstein-localized Hecke algebra $\bT_{\Gamma_1(N)}$ of level $\Gamma_1(N)$ is a $\Lambda$-algebra via diamond operators, and $\bT_{\Gamma_1(N)}^\circ$ is a free $\Lambda$-module with $\bT_{\Gamma_1(N)}^\circ \otimes_\Lambda \Z_p \cong \bT_{\Gamma_0(N)}^{\circ}$ by \cite[Thm.\ 3.2.2(4)]{HPWE}---see \S \ref{subsec: R=T conj} for details. Since, by assumption,  $\rk_{\Z_p} \bT_{\Gamma_0(N)}^{\circ} =2$ and thus $\rk_\Lambda \bT_{\Gamma_1(N)}^{\circ} =2$, there is a presentation
\begin{equation}
\label{eq: T1Ncirc form}
\bT_{\Gamma_1(N)}^\circ \cong \frac{\Lambda[y]}{(y^2+\lambda_1 y +\lambda_0)}
\end{equation}
for some $\lambda_0, \lambda_1 \in \m_\Lambda$.
Moreover, we show that the image of $J^\red$ in $\bT_{\Gamma_1(N)}^\circ/p\bT_{\Gamma_1(N)}^\circ$ is principal and that we can choose to map $y$ to a lift of a generator of $J^\red \cdot (\bT_{\Gamma_1(N)}^\circ/p\bT_{\Gamma_1(N)}^\circ)$, so that $\bT^\circ_{\Gamma_1(N)}/J^\red \bT^\circ_{\Gamma_1(N)} \cong \Lambda/(\lambda_0)$. Quotienting the presentation by $(p,yz,z^3)$ gives a surjective map
\begin{equation}
\label{eq: T1Ncirc surj}
\bT_{\Gamma_1(N)}^\circ \twoheadrightarrow\frac{\F_p[y,z]}{(z^3,y^2+\mu_1z+\mu_2z^2,yz)},
\end{equation} 
where $\lambda_0 \equiv \sum \mu_i z^i \in \F_p[z]/(z^3)$, and where $J^\red$ maps to the ideal generated by $y$. The structure of this quotient ring depends on the $\F_p$-dimension of the ring $\bar\bT_{\Gamma_1(N)}^{\circ,\red} := \bT_{\Gamma_1(N)}^{\circ}/(p,J^\red) \cong \Lambda/(p,\lambda_0)$.

The following cases could arise: 
\begin{enumerate}[label=(\roman*)]
\item \label{part: cases i}
If $\dim_{\F_p}\bar\bT_{\Gamma_1(N)}^{\circ,\red}=1$ (equivalently $\mu_1 \ne 0$), then the quotient ring in \eqref{eq: T1Ncirc surj} is $\F_p[y]/(y^3)$, and the techniques of \cite{part1} show that $\Cl_{T_0}(M)_{p,G} \to \Cl_{T}(M)_{p,G}$ is injective.
\item \label{part: cases ii}
If $\dim_{\F_p}\bar\bT_{\Gamma_1(N)}^{\circ,\red}=2$ (equivalently $\mu_1 = 0$ and $\mu_2 \ne 0$), then then the quotient ring in \eqref{eq: T1Ncirc surj} is $S$ (with $\mu=\mu_2$), so Theorem \ref{thm: main} shows that $\Cl_{T_0}(M)_{p,G} \to \Cl_{T}(M)_{p,G}$ is injective.
\item \label{part: cases iii}
If $\dim_{\F_p}\bar\bT_{\Gamma_1(N)}^{\circ,\red}>2$ (equivalently $\mu_1=\mu_2=0$), then $y^2$ is zero in the quotient ring in \eqref{eq: T1Ncirc surj}, so there is no surjection $\bT_{\Gamma_1(N)}^\circ  \rsurj S$.
\end{enumerate}

However, we expect that case \ref{part: cases ii} always occurs, 
and, in particular, that the fatal case \ref{part: cases iii} never occurs. 
We have two types of evidence for this. The first is theoretical: we sketch in Theorem \ref{thm: conjs imply result} that \ref{part: cases ii} follows from an ``$R=\bT$'' conjecture in \cite{HPWE}. A more detailed treatment of this argument, together with a proof of the ``$R=\bT$'' conjecture, are planned for a follow-up paper \cite{APAW} to \cite{HPWE}. The second type of evidence is computational, which we now explain.

\subsection{Computed examples}
\label{subsec: computation} We conclude this section by revisiting some of the specific cases studied in \cite{part2}. Indeed, with $p=5$ and $\ell_0=11$, there are four cases---$\ell_1=23,43,67,197$---that exhibit all of the interesting arithmetic properties appearing computationally in the results of \cite{part2}. For $\ell_1=43, 197$, we have $\rk_{\Z_p} \bT_{\Gamma_0(N)}^{\circ,\varep} >2$, so Eq. (\ref{eq: a1 at ell1 CFT}) is known.

For the remaining two cases $\ell_1=23,67$, we give a presentation below for the mod $p$ cuspidal Hecke algebra $\mathbb{T}^\circ_{\Gamma_1(N)}/p\mathbb{T}^\circ_{\Gamma_1(N)}$. Here $\mathbb{T}^\circ_{\Gamma_1(N)}$ is localization of the $\mathbb{Z}_p$-Hecke algebra acting on weight $2$ modular forms of level $\Gamma_1(N)$ at the Eisenstein maximal ideal containing $5$, $\langle d\rangle -1$, $U_{\ell_0}-1$, and $U_{\ell_1}-1$.  In each case, we obtain a presentation that gives the desired surjection in \eqref{eq: T1Ncirc surj}. All computations were completed using the ``Congruence subgroup $\Gamma_H(N)$" package within the SageMath \cite{SAGE} software system. 

\begin{eg} Let $\ell_1=23$, so $N=11\cdot 23$. We compute that $T_3$ and $\langle 3\rangle$
generate the localized cuspidal Hecke algebra $\mathbb{T}^\circ_{\Gamma_1(N)}$. In particular, as expected by the $\Lambda$-freeness result in \cite{HPWE} (see Lemma \ref{lem: freeness}), $\mathbb{T}^\circ_{\Gamma_1(N)}/5\mathbb{T}^\circ_{\Gamma_1(N)}$ has $\mathbb{F}_5$-dimension 10, and we compute an isomorphism
\[
\begin{split}
\frac{\mathbb{F}_5[y,z]}{(z^5,y^2-2z^2)}&\xrightarrow{\;\;\sim\;\;}\mathbb{T}^\circ_{\Gamma_1(N)}/5\mathbb{T}^\circ_{\Gamma_1(N)},\\
y&\;\longmapsto T_3\cdot\left(\langle 3\rangle^2-2\langle 3\rangle+2\right)^2 +2\langle 3\rangle^4+2\langle 3\rangle^3+\langle 3\rangle^2+2\langle 3\rangle+4,\\
z&\;\longmapsto \langle 3\rangle-1.
\end{split}
\]
\end{eg}

\begin{eg} Let $\ell_1=67$, so $N=11\cdot 67$. We again find that $T_3$ and $\langle 3\rangle$ generate the localized cuspidal Hecke algebra $\mathbb{T}^\circ_{\Gamma_1(N)}$ and that $\mathbb{T}^\circ_{\Gamma_1(N)}/5\mathbb{T}^\circ_{\Gamma_1(N)}$ has $\mathbb{F}_p$-dimension 10. We compute an isomorphism
\[
\begin{split}
\frac{\mathbb{F}_5[y,z]}{(z^5,y^2-z^2)}&\xrightarrow{\;\;\sim\;\;}\mathbb{T}^\circ_{\Gamma_1(N)}/5\mathbb{T}^\circ_{\Gamma_1(N)},\\
y&\;\longmapsto T_3\cdot \left(4\langle 3\rangle^4+4\langle 3\rangle^3+3\langle 3\rangle^2+\langle 3\rangle-1\right) \\
& \qquad \quad  +3\langle 3\rangle^4+2\langle 3\rangle^3+\langle 3\rangle^2+3\langle 3\rangle+2,\\
z&\;\longmapsto \langle 3 \rangle -1.\\
\end{split}
\]
\end{eg}

\subsection{Running assumptions}
\label{subsec: conventions}
For convenience, we concisely recall the running assumptions about the distinct prime numbers $p$, $\ell_0$, and $\ell_1$, since these are crucial for the validity of our arguments. We let $p,\ell_0 \geq 5$ be primes such that $\ell_0 \equiv 1 \pmod{p}$. As declared in \S\ref{subsec: intro add bad prime}, the additional running assumptions are that $\rk_{\Z_p} \bT_{\Gamma_0(\ell_0)}^{\circ,\varep} = 1$, that $\ell_1 \not\equiv \pm 1 \pmod{p}$, and that $\rk_{\Z_p} \bT_{\Gamma_0(N)}^{\circ,\varep} = 2$ (which requires $\ell_1$ to be a $p$th power modulo $\ell_0$).

\section{Algebraic number theory, Galois representations, and modular forms}
\label{sec: setup}

Our approach to investigating the main question in \S\ref{subsec: intro add bad prime} regarding an ideal class group of the number field $M$ involves a variety of tools, including the theory of cup products, Massey products of higher order, Galois (pseudo)representations, and modular forms. In this section, we set up the algebraic structures conveying these influences, beginning in \S \ref{subsec: intro to defs and Massey products} with a discussion of the overall philosophy, following \cite{WWE3, CarlA-inf}.

\subsection{Abstract-algebraic tools related to Ribet's lemma and deformation theory}
\label{subsec: intro to defs and Massey products}

This paper's main goal of establishing \eqref{eq: a1 at ell1 CFT} can be reframed in terms of understanding the quotient groups of $G = \Gal(\oQ/\Q)$. In this section, we give an overview of a method for studying this problem for a general group $G$, with a particular focus on studying its solvable quotients.

Let $G$ be a group, $F$ be a field, and $\rho : G \to \GL_d(F)$ be a homomorphism. One method for studying quotients of $G$ is to investigate deformations of $\rho$ to rings $F[\ep_i] := F[\ep]/(\ep^{i+1})$ for $i \geq 1$, i.e., $i$th-order deformations of $\rho$. For instance, the mere existence of a non-trivial deformation to $F[\ep_1]$---that is, a homomorphism $\rho_1 : G \to \GL_d(F[\ep_1])$ that reduces modulo $\ep$ to $\rho$, but does not arise from base change of $\rho$ along $F \to F[\ep_1]$---implies that there exists some non-zero map $\ker\rho \to M_{d}(F) \cong \ep M_d(F)$. Moreover, this map must be $G$-equivariant by $\rho$ in the sense that it is a 1-cocycle $e$ valued in $\ad \rho \cong M_d(F)$. Expressed coordinate-wise in a natural $F$-basis for $\rho_1$, we have 
\[
\rho_1 = \begin{pmatrix}
    \rho&\rho e\\
    & \rho
\end{pmatrix}.
\]

As noted in the context of Galois representations by Mazur \cite[\S1.6, Remark]{mazur1989}, the existence of a second-order deformation $\rho_2 : G \to \GL_d(F[\ep_2])$ such that $\rho_1 \simeq \rho_2$ (mod $\ep^2$) is equivalent to the vanishing of the cup product $e \cup e$ in cohomology $H^2(G, \ad\rho)$, where the ``cup product'' is the composition of the general cup product of cohomology with the multiplication map (thinking $\ad \rho \simeq M_d(F)$)
\[
H^1(G,\ad\rho) \otimes_F H^1(G,\ad\rho) \buildrel\cup\over\lra H^2(G, \ad\rho \otimes_F \ad\rho) \lra H^2(G, \ad\rho). 
\]
Again, considered coordinate-wise over $F$, $\rho_2$ has matrix expression
\[
\rho_2 = \begin{pmatrix}
        \rho & \rho e & \rho e_2 \\ 
        & \rho & \rho e \\
        & & \rho
    \end{pmatrix} : G \to \GL_{3d}(F),
\]
where  $-\!de_2 = e \smile e$ if and only if $\rho_2$ is a homomorphism, demonstrating that $\rho_2$ exists if and only if $e \cup e$ vanishes in $H^2(G,\ad\rho)$. Thus, the vanishing of the cup product is equivalent to the existence of a metabelian quotient of $\ker \rho$ realized by the upper-triangular coordinates of $\rho_2$. 

This phenomenon continues ad infinitum, as discussed in \cite{CarlA-inf}, and is a natural way to encounter Massey products of higher order. The existence of $\rho_3 : G \to \GL_d(F[\ep_3])$ lifting $\rho_2$ over the change of scalar ring $F[\ep_3] \rsurj F[\ep_2]$ is equivalent to the vanishing of the particular element of the (multi-valued, in general) Massey triple product $\langle e,e,e\rangle \in H^2(G,\ad\rho)$ given by $e \smile e_2 + e_2 \smile e$. (In this context, $e_2$ is an instance of a \textit{defining system} for a particular element of $\langle e,e,e\rangle$.) The proof of this equivalence may be drawn from the fact that 
\[
\rho_3 = \begin{pmatrix}
        \rho & \rho e & \rho e_2 & \rho e_3\\ 
        & \rho & \rho e & \rho e_2\\
        & & \rho & \rho e \\
        & & & \rho
    \end{pmatrix} : G \to \GL_{4d}(F), 
\]
where  $-\!de_3 = e \smile e_2 + e_2 + e$ if and only if $\rho_3$ is a homomorphism. 

While the above process links solvable quotients of $\ker \rho$ to representation theory, this paper, and its antecedent in Ribet's lemma \cite[\S2]{ribet1976}, add another wrinkle to the above picture; this wrinkle is \emph{lattice dependence}, although from the perspective of Ribet's lemma, this ``wrinkle'' is an asset. The representations $\tau : G \to \GL_2(\Q_p)$ that we consider in this paper (for $G = \Gal(\oQ/\Q)$) are \emph{residually reducible}. Twisting if needed, denote its residual semi-simplfication by $\chi \oplus 1 : G \to \F_p^\times \times \F_p^\times \subset \GL_2(\F_p)$. If $\tau$ is irreducible, there are multiple isomorphism classes of $G$-stable $\Z_p$-lattices $V$ in the representation space $\Q_p^{\oplus 2}$ of $\tau$. Ribet's lemma says that both directions of non-trivial extensions 
\[
\ttmat{\chi}{b}{0}{1} \qquad \text{and} \qquad \ttmat{\chi}{0}{\chi c}{1} : G \to \GL_2(\F_p)
\]
occur as $\bar V := V/pV$, as $V$ varies over $G$-stable $\Z_p$-lattices in $\rho$. Non-triviality of the extension is the equivalent to the 1-cocycle $b \in Z^1(G,\chi)$ (or $c \in Z^1(G,\chi^{-1})$) defining a non-zero cohomology class. Ribet applied this lemma to construct order $p$ quotients of $\ker \chi$. 

There are several approaches for combining the intricacies and power of lattice-dependence with the higher-order lifting process. For example, one method introduced by Skinner--Wiles \cite{SW1997} is to ``fix a preferred lattice,'' which works well when the chosen lattice selects a cocycle $b$ spans $H^1(G,\chi)$. Our approach is to apply the theory of (2-by-2) \emph{generalized matrix algebras} (GMAs) and the accompanying theory of \emph{pseudorepresentations}. A GMA is an $A$-algebra of the form $E = \sm{A}{B}{C}{A}$ where $A$ is a commutative ring and $B,C$ are $A$-modules. All of the coordinate-wise operations needed to define standard matrix multiplication are present already, except that the additional cross-diagonal multiplication map $m : B \otimes_A C \to A$ is needed. In particular, there is the usual matrix formula for the trace and determinant functions $E \to A$, and the resulting characteristic polynomials satisfy the Cayley--Hamilton theorem. 

A prime example of the additional flexibility gained by using GMAs within Ribet's method---even when the extension groups $H^1(G,\chi)$ and $H^1(G,\chi^{-1})$ have $\F_p$-dimension no more than $1$---is given by the construction of a GMA-valued $\rho_1$ below. It is valued in a GMA $E_1$, where for $i \geq 0$, the $\F_p[\ep_i]$-GMA $E_i$ discussed in \cite[\S4.1]{part1} is defined as
\begin{equation}
    \label{eq: define E_i}
    \F_p[\ep_i] := \F_p[\ep]/(\ep^{i+1}), \qquad E_i := \ttmat{\F_p[\ep_i]}{\F_p[\ep_{i-1}]}{\F_p[\ep_{i-1}]}{\F_p[\ep_i]} 
\end{equation}
whose cross-diagonal coordinate multiplication law is given by
\[
\F_p[\ep_{i-1}] \times \F_p[\ep_{i-1}] \to \F_p[\ep_i], \qquad (b,c) \mapsto \ep bc. 
\]
Let $b\up1 \in Z^1(G,\chi)$ (resp.\ $c\up1 \in Z^1(G,\chi^{-1})$) be a 1-cocycle whose cohomology class spans $H^1(G,\chi)$ (resp.\ $H^1(G,\chi^{1})$). Given the existence of some irreducible $\tau$ with residual semi-simplification $\rho = \chi \oplus 1$, Ribet's lemma provides that the classes of $b\up1$ and $c\up1$ are non-trivial and the associated cohomology groups have $\F_p$-dimension $1$, due to its construction of indecomposable representations
\[
\rho_b =\ttmat{\chi}{b\up1}{0}{1} \qquad \text{or} \qquad \rho_c = \ttmat{\chi}{0}{\chi c\up1}{1} : G \to \GL_2(\F_p)
\]

Now, to understand solvable quotients of $G$ beyond those constructed by Ribet's lemma, the classical approach to deforming representations would have us first pick one of $\rho_b$ and $\rho_c$ and then consider their lifts to $\GL_2(\F_p[\ep_i])$. For instance, deforming $\rho_b$ to first order amounts to finding 
\[
\rho_{b,1} = \ttmat{\omega(1 + \ep a\up1)}{b\up1 + \ep b_1}{\ep \omega c\up1}{1 + \ep d\up1} : 
G \to \GL_2(\F_p[\ep_1])
\]
(the $C$-coordinate must be filled by $c\up1$ up to natural isomorphism, provided that it is non-trivial) such that
\begin{equation}
    \label{eq: rho_b lift conditions}
\begin{split}
    \text{(i) }  -\!da\up1 = b\up1 \smile c\up1,  \quad 
    \text{(ii) } -\!dd\up1 = c\up1 \smile b\up1, \\ 
    \text{and (iii) }  -\!db_1 = a_1 \smile b\up1 + b\up1 \smile d\up1. 
\end{split}
\end{equation}
In other words, the cup products of (i) and (ii) and the element of the triple Massey product $\langle b\up1, c\up1,b\up1\rangle$ of (iii) must vanish in cohomology. 

Next, we demonstrate the flexibility of GMA-valued deformations. Non-trivially deforming $\chi \oplus 1 : G \to E_0^\times \subset \GL_2(\F_p)$ from $E_0 = \F_p \times \F_p$ to $\rho_1 : G_{\Q,\cS} \to E_1^\times$ over the natural projection $E_1 \rsurj E_0$ demands producing coordinate functions 
\begin{equation}
    \label{eq: abstract rho1}
    \rho_1 = \ttmat{\omega(1 + \ep a\up1)}{b\up1}{\omega c\up1}{1 + \ep d\up1}
\end{equation}
such that (i) and (ii) of \eqref{eq: rho_b lift conditions} holds. However, there is no need for (iii) of \eqref{eq: rho_b lift conditions} to hold as a result of the ``smaller'' $B$-coordinate module of $E_1$ (compared to $M_2(\F_p[\ep_1])$. 

To see the payoff of this simplicity---the lack of imposing (iii)---achieved by using $E_1^\times$ as a home for a deformation of $\omega \oplus 1$ in place of using $\GL_2(\F_p[\ep_1])$ to deform $\rho_b$, we need the theory of pseudorepresentations. 

\begin{defn}
    \label{def: PsR}
    A ($2$-dimensional) \emph{pseudorepresentation} of $G$ valued in a ring $A$, written $D : G \to A$, amounts to the data of a trace and determinant function $\Tr_D, {\det}_D : G \to A$ satisfying restrictions satisfied by the composition of a homomorphism $G \to \GL_2(A)$ with the usual matrix trace and determinant. We use the theory of pseudorepresentations and \textit{Cayley--Hamilton algebras} set up by Chenevier \cite{chen2014}.     
\end{defn}
While an $A$-valued pseudorepresentation of $G$ need not arise from an $A$-linear action of $G$ on a rank-two projective $A$-module in general, every pseudorepresentation of $G$ over $A$ \emph{does} arise as a composition of a representation valued in a Cayley--Hamilton $A$-algebra; in the setting of this paper, GMAs provide all needed Cayley--Hamilton algebras. 

In brief, as we explain more in \S\ref{subsec: attach Gal reps to eigenforms}, the Hecke algebras in this paper may not support a rank-two Galois representation, but they do support a pseudorepresentation. Thus, congruences among Hecke eigenforms give rise to deformations of Galois pseudorepresentations, thence deformations of GMA representations, thence the construction of number fields. 

\subsection{Algebraic number theory and Galois cohomology}
\label{subsec:ANT}

In this section, we relate the arithmetic condition of interest in this paper, expressed in \eqref{eq: a1 at ell1 CFT}, to cup and Massey products in Galois cohomology. 

For a group $G$, let $C^\bullet(G,-)$ be the complex of homogeneous cochains and $Z^\bullet(G,-)$ be the cocycles. Denote the cup product on $C^\bullet(G,-)$ by $\smile$ and the induced cup product on cohomology by $\cup$. For a cocycle $x \in Z^i(G,-)$, we let $[x] \in H^i(G,-)$ denote its cohomology class and write $x \cup y$ for the cohomology class $[x \smile y]=[x]\cup [y]$.

Let $\cS = \{\infty,p,\ell_0,\ell_1\}$, and, for a number field $K$, let $G_{K,\cS}$ be Galois group of the maximal extension of $K$ that is unramified outside the places of $K$ lying over those in $\cS$. Fix inertia and decomposition subgroups $I_q \subset G_q \rinj G_{\Q,\cS}$ for primes $q = p, \ell_0, \ell_1$, and let $\vert_{q}:C^\bullet(G_{\Q,\cS},-)\to C^\bullet (G_q,-)$ denote the restriction map. Fix a primitive $p$th root of unity $\zeta_p \in \ov\Q$; for each integer $i$, this defines an isomorphism $\mu_p^{\otimes i} \isoto \F_p$ of $\F_p$-vector spaces. We write $\F_p(i)$ the $G_{\Q,\cS}$-module $\F_p$ with this action. Also, write $\omega : G_{\Q,\cS} \to \F_p^\times$ for the action on $\F_p(1)$, the mod-$p$ cyclotomic character. 

\begin{lem}[{\cite[\S3.1]{part1}}]
There is a nonzero cocycle $c\up1 : G_{\Q,S} \to \F_p(-1)$ such that $c\up1|_{\ell_1}=0$ and $[c\up1|_{p}]=0$, and this cocycle is unique up to scalar. 
\end{lem}

The field $L/\Q(\zeta_p)$ from the introduction (\S\ref{subsec: intro Mazur}) is the splitting field of this cocycle $c\up1$. Since $c\up1 \vert_p$ is a coboundary, there is a constant $x_{c\up1} \in \F_p$ such that $x_{c\up1}=(\omega(\tau)-1)^{-1}c\up1(\tau)$ for every $\tau \in G_p$ with $\omega(\tau) \ne 1$. 

Let $\log_{\ell_0}: (\Z/\ell_0\Z)^\times \rsurj \F_p$ be a choice of discrete logarithm and let $a_0: G_{\Q,\cS} \rsurj \F_p$ be the composition $G_{\Q,\cS} \to \Gal(\Q(\zeta_{\ell_0})/\Q) \xrightarrow{\log_{\ell_0}} \F_p$. Let $a_p:  G_{\Q,\cS} \to \F_p$ be the composition $G_{\Q,\cS} \to \Gal(\Q_1/\Q) \simeq \F_p$, where $\Q_1 \subset \Q(\zeta_{p^2})$ is the degree-$p$ subextension. Then, by the Kronecker--Weber theorem, $a_0$ and $a_p$ span $Z^1(G_{\Q,\cS},\F_p)$. One of the main results of \cite{WWE3} is that $\bT_{\Gamma_1(\ell_0)}^\circ \simeq \Z_p$ if and only if $a_0 \cup c\up1 \ne 0$ if and only if $b\up1 \cup c\up1 \ne 0$. 

As in \S\ref{subsec: intro add bad prime}, define $M=L(\sqrt[p]{\ell_1})$ and let $G=\Gal(M/\Q)$. Let $b\up1 : G_{\Q,S} \to \F_p(1)$ be a Kummer cocycle of $\ell_1$. There is an isomorphism $\Gal(M/\Q(\zeta_p)) \isoto \F_p(1) \times \F_p(-1)$ given by $(b\up1,c\up1)$. By \cite[Lem.\ 3.2.1]{part1}, the assumption that $\ell_1$ is a $p$th power modulo $\ell_0$ implies that $b\up1 \cup c\up1=0$. Hence there exists a 1-cochain $a: G_{\Q,\cS}\to \F_p$ such that $-da=b\up1 \smile c\up1$, and this cochain is unique up to adding an element of $Z^1(G_{\Q,\cS},\F_p)$. We are interested in such $a$ because of their role in GMA-representations, as in \eqref{eq: abstract rho1}. The following lemma selects a particularly good choice of $a$.

\begin{lem}
    \label{lem: a1 setup} Assume that $a_0 \cup c\up1 \ne 0$. Then there exists a unique $1$-cochain $\mathbf{a}\up1 : G_{\Q,S} \to \F_p$ satisfying the following three properties:
    \begin{enumerate}[label=(\roman*)]
        \item\label{part: lem a1 part i} $-d \mathbf{a}\up1 = b\up1 \smile c\up1$,
        \item\label{part: lem a1 part ii} $(\mathbf{a}\up1+b\up1 \smile x_{c\up1})|_{I_p}=0$, and
        \item\label{part: lem a1 part iii} $\mathbf{a}\up1\vert_{\ell_0}$ is a multiple of $c\up1\vert_{\ell_0}$ in the $2$-dimensional $\F_p$-vector space $Z^1(G_{\ell_0},\F_p)$.
    \end{enumerate}  
    Moreover, the splitting field $M'/M$ of $\mathbf{a}\up1|_{G_{M,\cS}}$ is the maximal unramified elementary-abelian $p$-extension of $M$ that is Galois over $\Q$ and has all primes of $M$ over $\ell_0$ splitting completely. In other words, $\Gal(M'/M)$ is identified with $\Cl_{T_0}(M)_{p,G}$ by class field theory.
\end{lem}

The existence and uniqueness of the cochain $\mathbf{a}\up1$ is proven in \cite[Lem.\ 4.2.1]{part1}. The idea of the proof is that, since $b\up1 \cup c\up1=0$, there is a cochain satisfying \ref{part: lem a1 part i}, and it is unique up to adding multiplies of $a_p$ and $a_0$; the correct multiple of $a_p$ to add is determined by \ref{part: lem a1 part ii} and the correct multiple of $a_0$ to add is determined by \ref{part: lem a1 part iii}. The fact that $a_0 \cup c\up1$ is nonzero is used crucially to ensure that $a_0|_{\ell_0}$ and $c\up1|_{\ell_0}$ are not on the same line in $Z^1(G_{\ell_0},\F_p)$, so that is possible to achieve \ref{part: lem a1 part iii} by adding a multiple of $a_0$. The final statement is proven in \cite[Prop.\ 6.1.2]{part2}, and is essentially a translation of the uniqueness of $\mathbf{a}\up1$ in terms of algebraic number theory.

This lemma allows us to restate \eqref{eq: a1 at ell1 CFT} in terms of the cochain $\mathbf{a}\up1$. It is clear from the proof of the lemma that a cochain $a\up1$ satisfying properties \ref{part: lem a1 part i} and \ref{part: lem a1 part ii} is of the form $a\up1=\mathbf{a}\up1+\gamma a_0$ for some constant $\gamma \in \F_p$. It will be convenient for later applications to state the condition in terms of cochains of the form $a\up1$.

\begin{prop}
    \label{prop: a1 vanishing equivalences}
    Assume that $a_0 \cup c\up1 \ne 0$. Let $a\up1:G_{\Q,\cS} \to \F_p$ be a cochain satisfying properties (i) and (ii) of Lemma \ref{lem: a1 setup}, and let $d\up1=b\up1c\up1-a\up1$. The following conditions are equivalent:
    \begin{enumerate}[label=(\roman*)]
        \item\label{part: a1 prop 1} $a\up1\vert_{\ell_1} = 0$,
        \item\label{part: a1 prop 2} $\mathbf{a}\up1\vert_{\ell_1}=0$,
        \item\label{part: a1 prop 3} Every prime of $M$ lying over $\ell_1$ splits completely in $M'/M$,
        \item\label{part: a1 prop 4} The map $\Cl_{T_0}(M)_{p,G} \to \Cl_{T}(M)_{p,G}$ is injective,
        \item\label{part: a1 prop 5} The 2-cocycle $a\up1 \smile b\up1 + b\up1 \smile d\up1$ vanishes in $H^2(G_{\Q,\cS}, \F_p(1))$,
        \item\label{part: a1 prop 6} For some non-zero $b \in H^1(G_{\ell_1},\F_p(1))$, $a\up1\vert_{\ell_1} \cup b =0$.
    \end{enumerate}
\end{prop}

\begin{proof}
    Note that, since $\log_{\ell_0}(\ell_1)=0$ by assumption (\S\ref{subsec: conventions}), we have $a_0|_{\ell_1}=0$. Since $a\up1 =\mathbf{a}\up1+\gamma a_0$ for some $\gamma \in \F_p$, this makes the equivalence of \ref{part: a1 prop 1} and \ref{part: a1 prop 2} clear. The equivalence of \ref{part: a1 prop 2}, \ref{part: a1 prop 3}, and \ref{part: a1 prop 4} follows from the facts that $M'$ is the splitting field of $\mathbf{a}\up1$ and that $\Gal(M'/M)$ is identified with $\Cl_{T_0}(M)_{p,G}$ by class field theory. The spaces $H^1(G_{\ell_1},\F_p)=Z^1(G_{\ell_1},\F_p)$ and $H^1(G_{\ell_1},\F_p(1))$ both have dimension 1, and by Tate duality, the cup product pairing between them is perfect, so \ref{part: a1 prop 1} and \ref{part: a1 prop 6} are equivalent.

    It remains to show \ref{part: a1 prop 5} is equivalent to \ref{part: a1 prop 6}. Let $x$ denote the 2-cochain $a\up1 \smile b\up1 + b\up1 \smile d\up1$ in \ref{part: a1 prop 5}. Note that $-dd\up1=c\up1 \smile b\up1$, so $x$ is indeed a 2-cocycle by the Leibniz rule. By \cite[Lem.\ 3.2.8]{part1}, the sum of restriction maps
    \[
    H^2(G_{\Q,\cS}, \F_p(1)) \to H^2(G_{\ell_0},\F_p(1)) \oplus H^2(G_{\ell_1},\F_p(1))
    \]
    is injective. Since $b\up1\vert_{\ell_0}=0$, it is clear that $x \vert_{\ell_0}=0$, so \ref{part: a1 prop 5} is equivalent to $x\vert_{\ell_1}$ being a coboundary. Since $c\up1\vert_{\ell_1}=0$, we have $d\up1\vert_{\ell_1}=-a\up1\vert_{\ell_1}$, so 
    \[
    [x \vert_{\ell_1}] = a\up1\vert_{\ell_1} \cup b\up1\vert_{\ell_1} -b\up1\vert_{\ell_1} \cup a\up1\vert_{\ell_1} = 2 a\up1\vert_{\ell_1} \cup b\up1\vert_{\ell_1},
    \]
    by the skew-symmetry of the cup product. Since the class of $b\up1\vert_{\ell_1}$ spans $H^1(G_{\ell_1},\F_p(1))$, this shows that \ref{part: a1 prop 5} is equivalent to \ref{part: a1 prop 6}.
\end{proof}

In the language of Massey products, the class of the 2-cocycle $a\up1 \smile b\up1 + b\up1 \smile d\up1$ is a \emph{triple Massey product} of type $\langle b\up1,c\up1, b\up1 \rangle$ with respect to the defining system $a\up1$ and $d\up1$. In the strategy developed in \cite{WWE3,CarlA-inf} and summarized in \S\ref{subsec: intro to defs and Massey products}, it is Massey products of exactly this type that obstruct deforming certain Galois pseudorepresentations.

\subsection{The pseudodeformation ring $R$ and its tangent space} 
\label{subsec: set up R}
In this section, we introduce the pseudodeformation ring $R$. We survey some of the main ideas, referring to \cite[\S2]{part1} for a detailed discussion of pseudodeformation theory. 

Let $\bar D = \psi(\omega \oplus 1) : G_{\Q,\cS} \to \F_p$ be the residual pseudorepresentation associated to the weight $2$ Eisenstein series. For an Artinian local $\Z_p$-algebra $A$ with residue field $\F_p$, consider pseudorepresentations $D_A:G_{\Q,\cS} \to A$ such that
\begin{enumerate}[label=(\roman*)]
    \item $D_A \otimes_A \F_p = \bar D$, 
    \item $D_A$ is finite-flat at $p$, in the sense of \cite{WWE3,WWE4}.
\end{enumerate}
The functor sending $A$ to the set of $D_A$ satisfying these two conditions is pro-represented by a Noetherian local $\Z_p$-algebra $R$. Let $R_{\Gamma_0(N)}^\varep$ denote the quotient of $R$ corresponding to imposing the additional closed condition 
\begin{enumerate}[label=(\roman*),resume]
    \item $D_A$ is unramified-or-$\varepsilon$-Steinberg at $\ell_0$ and $\ell_1$, in the sense of \cite{WWE5}. 
\end{enumerate}

It is the ring $R_{\Gamma_0(N)}^\varep$ that is the central object of study in \cite{part1}; its structure is described in detail there. Although it is not stated in this way in \cite{part1}, many of the most basic structural results proven in \cite{part1} are applied only to $R_{\Gamma_0(N)}^\varep$ even though the identical proof also applies to $R$. We now restate the results we require about $R$ and refer to \cite{part1} for the proofs, which carry over to this setting verbatim.

There is a universal finite-flat at $p$ Cayley--Hamilton algebra $E$ over $R$, which is a GMA (as introduced in \S\ref{subsec: intro to defs and Massey products}) $E=\sm{R}{B}{C}{R}$ with cross-diagonal multiplication map $m: B \otimes_R C \to R$. The universality implies that there is a representation $\rho_R: G_{\Q,\cS} \to E^\times$ such that $\psi(\rho_R)$ is the universal pseudorepresentation $D_R$. 

The image of $m$ is an ideal $J^\red \subset R$ that is called \emph{reducibility ideal}; the quotient $R^\red =R/J^\red$ represents the functor of deformations $D_A$ that are reducible (i.e.~that can be realized as the sum of two characters). We let $\bar J^\red$ be the image of $J^\red$ in $R/pR$. (For all other $\Z_p$-modules $A$ in this paper, $\bar A = A/pA$; $J^\red$ is the only exception.) We call a tangent vector $\phi: R \to \F_p[\ep]/(\ep^2)$ of $R$ \emph{irreducible} if it does not factor through $R^\red$. 

We call a pseudorepresentation $D$ of $G_{\Q,\cS}$ \emph{unramified} at a prime $q$ if the inertia subgroup $I_q$ is contained in the kernel of $D$, defined in \cite[\S1.17]{chen2014}. This is also a Zariski-closed condition on pseudorepresentations; we call the corresponding ideal of $R$ its \emph{unramified-at-$q$} ideal. 

The best way to understand the modules $B$ and $C$ in terms of their connection with Galois cohomology. There are perfect pairings exactly as in \cite[Lem.\ 3.10.3]{WWE5}, 
\[
B/\m_R B \times H^1_f(G_{\Q,\cS},\F_p(1)) \to \F_p, \quad C/\m_R C \times H^1_f(G_{\Q,\cS},\F_p(-1)) \to \F_p,
\]
where the Selmer condition $f$ is at $p$ and comes from the finite-flat at $p$ condition on $R$. 
As in \cite[Prop.\ 2.4.2]{part1}, $H^1_f(G_{\Q,\cS},\F_p(-1))$ is spanned by the class of $c\up1$, and we let $c \in C$ be an element whose image in $C/\m C$ is a dual basis for $c\up1$. Similarly, $H^1_f(G_{\Q,\cS},\F_p(-1))$ is spanned by the class of $b\up1$ and the class of a Kummer cocycle $b\up0$ of $\ell_0$, and we let $b_0, b_1 \in B$ be elements that map in $B/\m B$ an ordered dual basis of $b\up0, b\up1$. Since $\{b_0,b_1\}$ and $c$ generate $B$ and $C$, respectively, as $R$-modules, it follows that $J^\red$ is generated by $b_0c$ and $b_1c$.

\begin{prop}
\label{prop:B and C} Assume that $b\up 0 \cup c\up1 \ne 0$. 
\begin{enumerate}[label=(\roman*)]
    \item \label{part: Jred gens} $\bar J^\red$ is principal and generated by $b_1c$, and not generated by $b_0c$.
    \item \label{part: ell1-unram gen} the unramified-at-$\ell_1$ ideal of $R$ is principal and generated by $b_1c$. 
\end{enumerate} 
\end{prop}

\begin{proof}
    The first point is as in \cite[Prop.\ 5.1.7]{part1}, using the fact that $b \up 0 \cup c\up1 \ne 0$. The second point is a standard consequence of the technique of imposing conditions on pseudorepresentations explained in \cite[\S2.4]{WWE4}, after using the fact that $R\cdot b_1 \subset B$ contains the image of $I_{\ell_1}$ in $E$. 
\end{proof}

\begin{lem}
\label{lem:rho1}
    Assume that $b\up 0 \cup c\up1 \ne 0$.
    Let $\phi: R \to \F_p[\ep]/(\ep^2)$ be an irreducible tangent vector of $R$ that factors through $R_{\Gamma_0(N)}^\varep$, and let $D_1:G_{\Q,\cS} \to \F_p[\ep]/(\ep^2)$ be the associated pseudodeformation. There is a homomorphism $\rho_1: G_{\Q,\cS} \to E_1^\times$ such that $D_1=\psi(\rho_1)$ and such that 
    \begin{equation}
    \label{eq:rho1}
            \rho_1(\sigma) = \ttmat{\omega(1+a\up1(\sigma)\ep)}{b\up1(\sigma)}{c\up1(\sigma)}{1+d\up1(\sigma)}
    \end{equation}
    where the cocycles $b\up1$ and $c\up1$ were defined as in \S\ref{subsec:ANT} and the cochains $a\up1$ and $d\up1$ as in Proposition \ref{prop: a1 vanishing equivalences}.
\end{lem}
\begin{proof}
    Such a homomorphism $\rho_1$ is constructed in \cite[Lem.\ 4.3.2]{part1} for a particular choice of irreducible tangent vector; for this $\rho_1$, the cochain $a\up1$ is the particular choice $\mathbf{a}\up1$ defined in Lemma \ref{lem: a1 setup}. As explained in \cite[\S5.1]{part1}, every irreducible tangent vector differs from this one by a reducible tangent vector up to scalar. By \cite[Lem.\ 2.3.2]{part1}, this difference arises from taking the pseudorepresentation associated to a new $\rho_1$ as in \eqref{eq:rho1} arising from adding $\gamma a_0$ to $\mathbf{a}\up1$ for some $\gamma \in \F_p$ (and subtracting $\gamma a_0$ from $d\up1$). 
\end{proof}

\section{Deformations to $S$ and the vanishing of a triple Massey product}
\label{sec: main proof}
In this section, we prove Theorem \ref{thm: main}. By Proposition \ref{prop: a1 vanishing equivalences}, it suffices to show that a certain triple Massey product vanishes. Following the general strategy of \cite{WWE3,CarlA-inf} summarized in \S\ref{subsec: intro to defs and Massey products}, this Massey-vanishing follows from the existence of a suitable pseudodeformation.

Let $S = \displaystyle\frac{\F_p[y,z]}{(y^3,y^2-\mu z^2,yz)}$ and let $\phi_S: S \to \F_p[\ep]/(\ep^2)$ be $\phi_S(y)=\ep$ and $\phi_S(z)=0$, as in the introduction. We define the $S$-GMA $E_S$ by 
\[
E_S = \ttmat{S}{B_S}{C_S}{S},
\qquad 
\text{where}
\qquad 
B_S = \F_p \oplus \F_p[\epsilon_1], \ C_S=\F_p[\epsilon_1] 
\]
are $S$-modules with $S$ acting via $\phi_S$. Like the construction of $E_1$ in \eqref{eq: define E_i}, the multiplication in $E_S$ is $2\times 2$-matrix multiplication, using the action of $S$ on $B_S$ and $C_S$ via $\phi_S$ and the cross-diagonal multiplication map $m:B_S \otimes_S C_S \to S$ given by
\begin{equation}
    \label{eq: E_S cross-diagonal}
    m : (u,v) \otimes x \mapsto v_0x_0 y + (\eta u x_0+v_0x_1+v_1x_0)y^2 
\end{equation}
where $v=v_0+v_1\epsilon$ and $x=x_0+x_1\epsilon$ ($u, v_0,v_1,x_0,x_1 \in \F_p$),  for some arbitrary fixed $\eta \in \F_p$. The accompanying $S$-algebra homomorphism lying over $\phi_S : S \to \F_p[\ep_1]$ is 
\[
\theta_S : E_S \to E_1 \quad \text{given by} \quad \ttmat{a}{(b',b_0+b_1\epsilon)}{c_0+c_1\epsilon}{d} \mapsto \ttmat{\phi_S(a)}{b_0}{c_0}{\phi_S(d)}.
\]

\begin{proof}[Proof of Theorem \ref{thm: main}] Recall that the assumption $\rk_{\Z_p}\bT^{\circ,\varep}_{\Gamma_0(\ell_0)}=1$ implies that $a_0 \cup c\up1$ and $b\up1 \cup c\up1$ are nonzero, and hence we can apply the results of Section 2.
Let $\pi: R \rsurj S$ be such that $\pi(J^\red)=yS$ and $\phi_S \circ \pi$ factors through $R_{\Gamma_0(N)}^\varep$. This implies that $\phi_S \circ \pi$ is irreducible; let $\rho_1: G_{\Q,\cS} \to E_1^\times$ be the representation associated to it by Lemma \ref{lem:rho1}. By Lemma \ref{prop:B and C}, we have $\pi(b_1c)S=yS$ and $\pi(b_0c)S \subset y \m_S =\F_p y^2$.  We let the constant $\eta \in \F_p$ in the definition of $E_S$ be such that $\pi(b_0c)=\eta y^2$.

We will first construct an $S$-algebra homomorphism $E \otimes_R S \to E_S$; for this we repeat the argument of \cite[Lemma 6.1.3]{part1}. We first construct it as a map of $S$-modules, which amounts to constructing maps $B \otimes_R S \to B_S$ and $C \otimes_R S \to C_S$. The simpler map is $C \otimes_R S \to C_S$. Since $c$ generates $C$, there is a surjective map
\[
\pi_C: S \twoheadrightarrow C \otimes_R S, \quad \pi_C(s)=sc.
\]
If $s \in \ker(\pi_C)$, then $sc=0$, so $s \in \Ann_S(b_1c)=\Ann_S(y)=\ker(\phi_S)$. Hence, we can define the map $C \otimes_R S \to C_S$ by $sc \mapsto \phi(s)$. The construction of $B \otimes_R S \to B_S$ is similar, but more involved.
Since $\{b_0,b_1\}$ is a minimal generating set for $B$, there is a surjective map
\[
\pi_B:S^{\oplus 2} \twoheadrightarrow B \otimes_R S, \quad \pi_B(s_0,s_1)=s_0b_0+s_1b_1.
\]
Define a map
\[
\pi_{B_S}:S^{\oplus 2} \twoheadrightarrow B_S, \quad \pi_{B_S}(s_0,s_1)=(s_0 \bmod \m_S, \phi_S(s_1)).
\]
We will show that $\pi_{B_S}$ factors through $\pi_B$, thus defining a map $B \otimes_R S \to B_S$. To this end, let $(s_0,s_1) \in \ker(\pi_B)$. Since $B$ is not cyclic, we have $s_0,s_1 \in \m_S$, so $\pi_{B_S}(s_0,s_1)=(0,\phi(s_1))$, and we need to show that $s_1 \in \ker(\phi_S)=\Ann_S(b_1c)$.  Since $(s_0,s_1) \in \ker(\pi_B)$, we have $s_0b_0+s_1b_1=0$, which, multiplying by $c$, gives
\[
0=s_0b_0c+s_1b_1c.
\]
As $b_0c \in y\m_S$, we have $s_0b_0c \in y \m_S^2=0$, so the above equation simplifies to $0=s_1b_1c$. In other words, $s_1 \in \Ann_S(b_1c)$, as desired. Finally, we have to show that the $S$-module map $E \otimes_R S \to E_S$ is a ring homomorphism. This follows easily from the definitions, just as in the proof of \cite[Lem.\ 6.1.3]{part1}, using the definition of $\eta$.

Now, let $\rho_S: G_{\Q,\cS} \to E_S^\times$ be the composition of the universal map $\rho_R:G_{\Q,\cS} \to E^\times$ with the map $E \otimes_R S \to E_S$ constructed in the previous paragraph. The existence of this representation $\rho_S$ implies the vanishing of the required triple Massey product, just as in the proof of \cite[Lem.\ 6.2.9]{part1}. Indeed, by construction, $\theta_S \circ \rho = \rho_1$. Writing $\rho_S$ in coordinates as
\[
\rho_S = \ttmat{a_S}{b_S}{c_S}{d_S}
\]
we have $\phi_S \circ a_S=\omega(1+a\up1\ep)$ and $\phi_S \circ d_S = 1+d\up1 \ep$. Examining the $B$-coordinate, the condition that $\rho_S$ be a homomorphism implies that, for all $\sigma,\tau \in G_{\Q,\cS}$,
\begin{equation}
\label{eq:b_S hom}
  b_S(\sigma\tau) = \phi_S(a_S(\sigma))b_S(\tau) +b_S(\sigma)\phi_S(d_S(\tau)).
\end{equation}
For $\sigma \in G_{\Q,\cS}$, the element $b_S(\sigma)$ can be written as 
\[
b_S(\sigma)=(b'(\sigma),b\up1(\sigma)+b\up2(\sigma)\epsilon) \in B = \F_p \oplus \F_p[\ep_1].
\]
Thinking of $b\up2$ as a cochain $b\up2:G_\Q \to \F_p(1)$,  \eqref{eq:b_S hom} can be written in terms of coboundary equations and implies
\[
-d b\up2 = a\up1 \smile b\up1 +b\up1 \smile d\up 1,  
\]
and consequently 
\[
a\up1 \smile b\up1 +b\up1 \smile d\up 1 \text{ is a coboundary}. 
\]
Thus the map $\Cl_{T_0}(M)_{p,G} \to \Cl_{T}(M)_{p,G}$ of \eqref{eq: a1 at ell1 CFT} is injective by Proposition \ref{prop: a1 vanishing equivalences}.
\end{proof}

\section{Hecke algebras: irreducibility of a cuspidal tangent vector}
\label{sec: Hecke algs}
For the rest of the paper, we focus on a strategy for constructing the putative map $\pi: R \rsurj S$.
Our strategy is to show $\bT_{\Gamma_1(N)}^\circ$ is ``large enough'' that it admits a map $\pi' : \bT_{\Gamma_1(N)}^\circ \rsurj S$ such that the resulting $S$-valued pseudorepresentation produces a surjection $\pi: R \rsurj S$ satisfying the hypotheses of Theorem \ref{thm: main}. We summarize it in two steps:
\begin{enumerate}
\item The main result in this section (Proposition \ref{prop: cusp Gamma0(N) irreducible}) establishes that the Hecke algebra quotient $\bT^{\circ}_{\Gamma_1(N)} \rsurj \bT^{\circ}_{\Gamma_0(N)}/p\bT^{\circ}_{\Gamma_0(N)} \simeq \F_p[\ep_1]$ carries a pseudorepresentation  corresponding to an irreducible tangent vector in the sense of Lemma \ref{lem:rho1}. 
\item The main result in $\S\ref{sec:R=T}$ (Theorem \ref{thm: conjs imply result}) explains how the $R=\mathbb T$ conjectures formulated in \cite{HPWE} imply that the above map $\bT^{\circ}_{\Gamma_1(N)}\rsurj \F_p[\ep_1]$ can be lifted to a quotient isomorphic to $S$. 
\end{enumerate}
As highlighted in \S\ref{subsec: intro relax ell0}, this strategy crucially relies on using a larger, $\Lambda$-free Hecke algebra of level $\Gamma_1(N)$ since the cuspidal Hecke algebra of level $\Gamma_0(N)$ is not always sufficiently large to produce the Galois representations required to establish \eqref{eq: a1 at ell1 CFT}. However, a technical difficulty arises related to the choice of ramified Hecke operators: the approach in \cite{part1}, which was expanded upon to prove all results in \S\S\ref{sec: setup}--\ref{sec: main proof}, uses Atkin--Lehner operators $\mathbf{w}_\ell$; while there are substantial advantages to using an Atkin--Lehner localization---see \cite[\S1.4.3]{WWE5}---it is inconvenient to extend it to levels $\Gamma_1(N)$ because Atkin--Lehner operators $\mathbf{w}_\ell$ do not generally commute with non-trivial diamond actions of $\F_\ell^\times$. 

We begin this section by giving a brief overview of the Hecke algebras under consideration and the two choices for ramified Hecke operators; after the overview, we proceed with Step 1 in the strategy above.

\subsection{Hecke algebras and ramified Hecke operators}
\label{subsec: Hecke U vs w}

For a squarefree integer $n$ with $p \nmid n$ and $j = 0,1$, let $M_{\Gamma_j(n)}$ (resp. $S_{\Gamma_j(n)}$ denote the space of modular forms (resp. cusp forms) of weight $2$ and level $\Gamma_j(n)$ with $\Z_p$-integral $q$-expansions at the cusp $\infty$.
The Hecke algebras discussed in this paper are defined over $\Z_p$ as the image of Hecke operators acting on the $\Z_p$-modules $M_{\Gamma_j(n)}$ of modular forms for $n = \ell_0, N$; they are generated over $\Z_p$ by unramified Hecke operators $T_m$ for $(m,n) = 1$, diamond operators $\dia{d}$ for $d \in (\Z/n\Z)^\times$, and some choice of ramified Hecke operators at primes $\ell \mid n$. More specifically, for $n = \ell_0, N$, we consider the following two choices of ramified Hecke operators and localizations at appropriate Eisenstein ideals. 
\begin{itemize}[leftmargin=1.1em]
    \item As in \S\ref{subsec: intro Mazur}, $\bT_{\Gamma_0(n)}^\varep$ denotes the Hecke algebra that includes the Atkin-Lehner involutions $\mathbf{w}_\ell$ for $\ell\,|\,n$ and has been localized at the Eisenstein maximal ideal generated by $p$ and, letting $\sigma$ denote the sum-of-divisors function,  
    \begin{equation} \label{eq: eigensystem w}
T_m -\sigma(m) \text{ for } (m,n) = 1, \text{ and } \mathbf{w}_\ell + 1 \text{ for primes }\ell \mid n. 
\end{equation}
This choice of Eisenstein localization is not necessary for the arguments in \S\S\ref{subsec: stabilization and cusps}-\ref{subsec: cuspidal tangent irreducible conclusion}; we do this to ensure consistent and minimally complicated notation. 
\item For $j=0,1$, $\bT_{\Gamma_j(n)}$ denotes the Hecke algebra that includes $T_\ell$ for primes $\ell \mid n$---as is common, we denote them by $U_\ell$---and has been localized at the Eisenstein maximal ideal generated by $p$ along with 
\begin{equation} \label{eq: eigensystem}
T_m -\sigma(m) \text{ for } (m,n) = 1, \dia{d} = 1, U_\ell - 1 \text{ for primes }\ell \mid n. 
\end{equation}
Each of these Hecke algebras admits a natural $\Lambda$-algebra structure arising from the diamond action (where $\Lambda$ was defined in \S\ref{subsec: intro tame family}), and by Lemma \ref{lem: freeness} we have $\bT_{\Gamma_1(n)} \otimes_\Lambda \Lambda/z\Lambda \cong \bT_{\Gamma_0(n)}$ and $\bT_{\Gamma_1(n)}^\circ \otimes_\Lambda \Lambda/z\Lambda \cong \bT_{\Gamma_0(n)}^\circ$. 
\end{itemize}
For $\bT_{\Gamma_0(n)}^\varep$ or $\bT_{\Gamma_j(n)}$, denote their quotient acting faithfully on cusp forms by $\bT_{\Gamma_0(n)}^{\circ,\varep}$ or $\bT_{\Gamma_j(n)}^\circ$, respectively. Since $z \in \Lambda$ generates its augmentation ideal, there is an isomorphism $\Lambda/z\Lambda \cong \Z_p$. Henceforth, we write $- \otimes_\Lambda \Z_p$ for $- \otimes_\Lambda \Lambda/z\Lambda$.

Now, for level $\Gamma_0(n)$, the Hecke algebras $\bT_{\Gamma_0(n)}^\varep$ and $\bT_{\Gamma_0(n)}$ are often closely related. Specifically, when acting on eigenforms that are new at a prime $\ell\mid N$, they satisfy the relation $\mathbf{w}_{\ell}=-U_{\ell}$, meaning the localization at $\mathbf{w}_{\ell} = -1$ matches the localization at $U_\ell = 1$. On eigenforms that are old at $\ell$, a discrepancy arises essentially when the Hecke polynomial at $\ell$ has repeated roots mod $p$. In general, the passage between $\mathbf{w}_\ell$- and $U_\ell$-operators is a subtle issue, which is explored in the upcoming paper \cite{APAW} in order to 
extend certain $R = \bT$ results to the level $\Gamma_1(n)$-setting. 

Fortunately, for the purposes of this paper, there is a simple resolution of this difference: we only need to pass from $\mathbf{w}$-operators to $U$-operators when considering the cuspidal Hecke algebras $\bT_{\Gamma_0(n)}^{\circ,\varep}$ and $\bT_{\Gamma_j(n)}^\circ$. The following lemma establishes a natural isomorphism between them. 

\begin{lem}
\label{lem: natural iso} 
Under the running assumptions (\S\ref{subsec: conventions}), there is an isomorphism
\[
\bT_{\Gamma_0(N)}^\circ\simeq\bT_{\Gamma_0(N)}^{\circ,\varep}\]
characterized by $T_{m}\mapsto T_{m}$ for every $(m,N)=1$. 
\end{lem}
\begin{proof}
Following the argument in \cite[App.\ A]{WWE5}, we can consider normalization maps for  $\bT_{\Gamma_0(N)}^\circ$ and $\bT_{\Gamma_0(N)}^{\circ,\varep}$. The key points that make this comparison straightforward is that both of these newforms are $\ell_0$-new (so $U_{\ell_0}$ is constant), and that $\ell_1\not\equiv 1\pmod{p}$ (so $U_{\ell_1}$ is contained in the subalgebra generated by unramified Hecke operators). 
\end{proof}

\begin{rem}\label{rem: Hecke w vs U}Going forward, it will be important to keep track of the choice of ramified Hecke operators in Hecke algebras of levels $\Gamma_0(n)$; all Hecke algebras of level $\Gamma_1(n)$ include $U_\ell$.  For the sake of clarity, we emphasize that the assumption in \eqref{eq: min rk} that $\rk_{\Z_p} \bT_{\Gamma_0(N)}^\circ = 2$ implies that
$\rk_{\Z_p} \bT_{\Gamma_0(N)}^\varep = 3$ and $\rk_{\Z_p} \bT_{\Gamma_0(N)} = 4$.\end{rem}

\subsection{Galois pseudorepresentations attached to Hecke eigenforms}
\label{subsec: attach Gal reps to eigenforms}
As is well-known, a normalized modular Hecke eigenform $f$ of weight 2 and level $\Gamma_1(n)$ with $q$-series expansion $f = \sum_{i \geq 0} a_i(f)q^i$ and nebencharacter $\chi$ gives rise to Galois representations $\rho_f : G_{\Q,\cS} \to \GL_2(\oQ_p)$ characterized up to semi-simplification by $\Tr \rho_f(\Fr_q) = a_q(f)$, $\det(\Fr_q) = q\chi(q)$ for all primes $q \nmid Np$; here $\Fr_q$ denotes a choice of Frobenius element at $q$. Over non-Eisenstein congruence classes, one can show that these representations interpolate to a two-dimensional Galois representation valued in the corresponding Hecke algebra. In contrast, for Eisenstein congruence classes, only the trace and determinant functions of these representations interpolate well in general. In other words, taking either $\bT=\bT_{\Gamma_j(n)}$ ($j=0,1$) or $\bT=\bT_{\Gamma_0(n)}^\varep$,
we can make sense of a pseudorepresentation $D : G_{\Q,\cS} \to \bT$ such that 
\[
\Tr_D(\Fr_q) = T_q, \qquad {\det}_D(\Fr_q) = q\dia{q}.
\]
This interpolates the pseudorepresentations $\psi(\rho_f)$ associated to the Hecke eigenforms $f$ parameterized by $\bT$: we have $\pi_f \circ D = \psi(\rho_f)$, where $\pi_f : \bT\to \oQ_p$ encodes the Hecke eigensystem of $f$. 

Because all the Hecke algebras we consider are localized at the unramified Eisenstein eigensystem 
\[
T_m \mapsto \sigma(m) \text{ for } (m,n) = 1, \;\;\;\dia{d} \mapsto 1,
\]
their associated pseudorepresentation is a deformation of the fixed pseudorepresentation $\Db = \psi(\omega \oplus 1) : G_{\Q,\cS} \to \F_p$. Moreover, 
$\rho_f\vert_p$ is finite-flat (see e.g.\ \cite[Thm.\ 3.1(f), pg.~86]{DDT1994}). Thus, by the theory of conditions on pseudorepresentations set up in \cite{WWE4} and applied here in \S\ref{subsec: set up R} to construct the universal finite-flat-at-$p$ deformation $D_R : G_{\Q,\cS} \to R$ of $\Db$, the pseudorepresentation $D : G_{\Q,\cS} \to \bT$ gives rise to a local homomorphism $R \to \bT$. From now on, we consider all Hecke algebras defined in \S\ref{sec: Hecke algs} to be $R$-algebras.

The running assumption in \eqref{eq: min rk} that $\rk_{\Z_p} \bT^\circ_{\Gamma_0(N)}=2$ implies that there exists an isomorphism 
\[
\bar \bT^\circ_{\Gamma_0(N)} \simeq \F_p[\ep_1]=\F_p[\ep]/(\ep^2).
\]
Thus, the map $R \to \bar\bT^\circ_{\Gamma_0(N)}$ can be thought of as a tangent vector for $R$ up to $\F_p^\times$-scalar. In \S\S\ref{subsec: stabilization and cusps}--\ref{subsec: cuspidal tangent irreducible conclusion}, we prove:

\begin{prop}
    \label{prop: cusp Gamma0(N) irreducible}
    The tangent vector $R \to \bar\bT^\circ_{\Gamma_0(N)} \simeq \F_p[\ep_1]$ is irreducible in the sense explained in \S\ref{subsec: set up R}. 
\end{prop}
Note that, in particular,  as the resulting tangent vector is non-zero, this implies that $R\to \bT^{\circ}_{\Gamma_0(N)}$ is surjective. The surjectivity of this map that is equivalent to $\bT_{\Gamma_0(N)}^{\circ}$ being generated by unramified Hecke operators.

\subsection{Behavior of cuspidal quotients with respect to stabilization}
\label{subsec: stabilization and cusps}
Let $S_{\Gamma_0(n)}^\varep \subset M_{\Gamma_0(n)}^\varep$ represent the respective subspaces of $S_{\Gamma_j(n)}\subset M_{\Gamma_j(n)}$ such that the Atkin--Lehner involutions $\mathbf{w}_\ell$ acts by $-1$ for all primes $\ell \mid n$ and the unramified Hecke action is localized at the Eisenstein eigensystem $T_m \mapsto \sigma(m)$ for $(m,n)=1$. 
As in the proof of \cite[Lem.\ 3.2.5]{ohta2014}, the fact that the Atkin--Lehner group acts transitively on the set of cusps of the closed modular curve $X_0(n)$ over $\Z_p$ of level $\Gamma_0(n)$ implies that there is a short exact sequence
\begin{equation}
\label{eq:res seq}
  0 \to S_{\Gamma_0(n)}^\varep\to M_{\Gamma_0(n)}^\varep\xrightarrow{a_0} \Z_p \to 0,  
\end{equation}
where $a_0(f)$ is the constant term at the infinity cusp. As in \cite[Thm.\ 2.4.6]{ohta2014}, there is a perfect duality between $M_{\Gamma_0(n)}^\varep$ and $\bT_{\Gamma_0(n)}^\varep$ given by $(f,t)\mapsto a_1(tf)$, and similarly for cusp forms. By \eqref{eq:res seq}, this implies that there is a unique element $t_0 \in \bT_{\Gamma_0(n)}^\varep$ such that $a_1(t_0f)=a_0(f)$ for all $f \in M_{\Gamma_0(n)}$, and that $\bT^{\circ,\varep}_{\Gamma_0(n)}=\bT^{\varep}_{\Gamma_0(n)}/(t_0)$. 

Now, let $\ell$ be a prime with $\ell \nmid n$. Let $\iota: M_{\Gamma_0(n)} \to M_{\Gamma_0(n\ell)}$ denote the inclusion that is the identity on $q$-expansions. This induces an inclusion $\iota^\varep: M_{\Gamma_0(n)}^\varep\to M_{\Gamma_0(n\ell)}^\varep$ defined by $\iota^\varep(f)=\iota(f)-\mathbf{w}_\ell \circ\iota(f)$. As a function of $\tau$ in the upper-halfplane, one can easily compute that $\iota^\varep(f)(\tau)=f(\tau)-\ell f(\ell\tau)$. In particular, $a_0(\iota^\varep(f))=(1-\ell)a_0(f)$, so the sequences \eqref{eq:res seq} for $n$ and $n\ell$ fit into a commutative diagram 
\[
\xymatrix@R=1.2em{
0 \ar[r] & S_{\Gamma_0(n)}^\varep\ar[r] \ar[d]^-{\iota^\varep} &
M_{\Gamma_0(n)}^\varep\ar[r]^-{a_0} \ar[d]^-{\iota^\varep} &
\Z_p \ar[r] \ar[d]^-{1-\ell} & 0 \\
0 \ar[r] & S_{\Gamma_0(n\ell)}^\varep\ar[r]  &
M_{\Gamma_0(n\ell)}^\varep\ar[r]^-{a_0}  &
\Z_p \ar[r] & 0.
}
\]
If $\pi^\varep: \bT_{\Gamma_0(n\ell)}^{\varep} \to \bT_{\Gamma_0(n)}^{\varep}$ denotes the map dual to $\iota^\varep$, then the diagram implies that $\pi^\varep(t_0)=(1-\ell)t_0$, so that
\[
\bT_{\Gamma_0(n)}^{\varep} \otimes_{\bT_{\Gamma_0(n\ell)}^{\varep}} \bT_{\Gamma_0(n\ell)}^{\circ,\varep} = \bT_{\Gamma_0(n)}^{\varep}/(\pi^\varep(t_0)) =\bT_{\Gamma_0(n)}^{\varep}/((1-\ell)t_0).
\]
In particular, we will use the following consequence of the calculation:
\begin{lem}
    \label{lem: transverse old and cusp}
    In the setting described above, if $\ell \not \equiv 1 \pmod{p}$, then $\bT_{\Gamma_0(n)}^{\varep} \otimes_{\bT_{\Gamma_0(n\ell)}^{\varep}} \bT_{\Gamma_0(n\ell)}^{\circ,\varep} \cong \bT_{\Gamma_0(n)}^{\circ,\varep}$.
\end{lem}

\subsection{The cuspidal tangent vector is irreducible} 
\label{subsec: cuspidal tangent irreducible conclusion}
We now specialize the discussion to the case of interest in this paper where $n=\ell_0$ and $\ell=\ell_1$ (so that $n\ell=N$). 
Since $\ell_1 \not \equiv 1 \pmod{p}$, we have the isomorphism $\bT_{\Gamma_0(\ell_0)}^{\varep} \otimes_{\bT_{\Gamma_0(N)}^{\varep}} \bT_{\Gamma_0(N)}^{\circ,\varep} \cong \bT_{\Gamma_0(\ell_0)}^{\circ,\varep}$ from Lemma \ref{lem: transverse old and cusp}. Since the product map
\[
\bT_{\Gamma_0(N)}^{\varep} \to \bT_{\Gamma_0(\ell_0)}^{\varep} \times \bT_{\Gamma_0(N)}^{\circ,\varep}
\]
is injective (this is easily checked after inverting $p$) and the image is the fiber product over the maximal common quotient $\bT_{\Gamma_0(\ell_0)}^{\varep} \otimes_{\bT_{\Gamma_0(N)}^{\varep}} \bT_{\Gamma_0(N)}^{\circ,\varep}$, Lemma \ref{lem: transverse old and cusp} implies that this product map induces isomorphisms
\begin{equation}
    \label{eq: Hecke transverse}
    \text{(i) } \bT_{\Gamma_0(N)}^{\varep} \isoto \bT_{\Gamma_0(\ell_0)}^{\varep} \times_{\bT_{\Gamma_0(\ell_0)}^{\circ,\varep}} \bT_{\Gamma_0(N)}^{\circ,\varep}
    \qquad \text{and} \qquad 
    \text{(ii) } 
    \bar\bT_{\Gamma_0(N)}^{\varep} \isoto \bar\bT_{\Gamma_0(\ell_0)}^{\varep} \times_{\bar\bT_{\Gamma_0(\ell_0)}^{\circ,\varep}} \bar\bT_{\Gamma_0(N)}^{\circ,\varep},
\end{equation}
(where $\bar A = A/pA$). The left isomorphism induces the right isomorphism because $\bT_{\Gamma_0(\ell_0)}^{\circ,\varep}$ is $p$-torsion free.

Because $\rk_{\Z_p}\bT_{\Gamma_0(N)}^{\circ,\varep}=2$ and  
$\rk_{\Z_p} \bT_{\Gamma_0(\ell_0)}^{\circ,\varep} = 1$ (\S\ref{subsec: conventions}), and because there is only one $\Z_p$-valued normalized Eisenstein series of these levels with all Atkin--Lehner eigenvalues are $-1$, we deduce that $\rk_{\Z_p} \bT_{\Gamma_0(N)}^\varep = 3$ and $\rk_{\Z_p} \bT_{\Gamma_0(\ell_0)}^{\varep} = 2$. In particular, both $\bar\bT_{\Gamma_0(\ell_0)}^{\varep}$ and $\bar\bT_{\Gamma_0(N)}^{\circ,\varep}$ are isomorphic to $\F_p[\epsilon_1]$; thus the surjections $\bar\bT_{\Gamma_0(N)}^{\varep} \to \bar\bT_{\Gamma_0(\ell_0)}^{\varep}$ and $\bar\bT_{\Gamma_0(N)}^{\varep} \to \bar\bT_{\Gamma_0(N)}^{\circ,\varep}$ can be thought of as tangent vectors. The isomorphism 
of \eqref{eq: Hecke transverse}(ii) 
implies that these two tangent vectors span the tangent space of $\bar\bT_{\Gamma_0(N)}^\varep$. As in \cite[\S5]{WWE3}, the map $R \rsurj \bar\bT_{\Gamma_0(\ell_0)}^{\varep} \simeq \F_p[\epsilon_1]$ is a reducible tangent vector, 
in the sense discussed in \S\ref{subsec: set up R}. Since the reducible subspace of the tangent space of $\bT_{\Gamma_0(N)}^{\varep}$ is $1$-dimensional by \cite[Prop.\ 5.1.8]{part1}, this implies that the tangent vector $R \rsurj \bar\bT_{\Gamma_0(N)}^{\circ,\varep} \simeq \F_p[\epsilon_1]$ is irreducible. By Lemma \ref{lem: natural iso}, we have $\bar\bT_{\Gamma_0(N)}^{\circ,\varep} \cong \bar\bT_{\Gamma_0(N)}^{\circ}$, so this completes the proof of Proposition \ref{prop: cusp Gamma0(N) irreducible}.

\section{Structure of the Hecke algebra $\bT_{\Gamma_1(N)}^{\circ}$}
\label{sec:R=T}

We conclude this paper with an argument justifying the expectation that the Hecke algebra $\bT_{\Gamma_1(N)}^{\circ}$ has a quotient isomorphic to the ring $S$ as in Theorem \ref{thm: main}. This argument uses results on the structure of the Hecke algebras of level $\Gamma_1(n)$ for $n$ squarefree obtained in \cite{HPWE} as well as some instances of $R = \bT$-style conjectures also formulated in \cite{HPWE}. We discuss these inputs in \S \ref{subsec: R=T conj} and draw implications for the structure of $\bT_{\Gamma_1(N)}^{\circ}$ in \S \ref{subsec: cuspidal quotient}. 

\subsection {Hecke algebras as $\Lambda$-module and $R=\bT$ conjectures at level $\Gamma_1(n)$}
\label{subsec: R=T conj}

Hecke algebras of level $\Gamma_1(n)$ can be viewed naturally as algebras over group rings via the action of diamond operators.  For non-Eisenstein congruence classes, the group-ring module structure has been widely studied. For example, the Taylor--Wiles method relies on certain freeness results. By contrast, for Eisenstein congruence classes, similar results often hold for Hecke algebras acting faithfully on spaces of modular forms, but in weight $2$ they sometimes fail for cusp forms (see Lemma \ref{lem: freeness}(iii)). 

Recall that, for $j=0,1$ and the algebras $\bT_{\Gamma_j(n)}$ and their cuspidal quotients $\bT_{\Gamma_j(n)}^{\circ}$ are obtained as localizations corresponding to the mod $p$-system of Hecke eigenvalues defined by \eqref{eq: eigensystem}.
Also, recall that  $\Lambda=\Z_p[\Delta]$, where $\Delta$ is the $p$-Sylow subgroup of $(\Z/\ell_0\Z)^{\times}$, or equivalently, of $(\Z/N\Z)^{\times}$, under our running assumptions.
For $n \in \{\ell_0, 
\ell_1\}$, restricting the Hecke action from cusp forms of level $\Gamma_1(n)$ to $\Gamma_0(n)$ gives canonical surjections
\begin{equation}
\label{eq: augmentation quotient}
\bT_{\Gamma_1(n)} \otimes_{\Lambda} \Z_p \twoheadrightarrow \bT_{\Gamma_0(n)}, 
\qquad
\bT_{\Gamma_1(n)}^{\circ} \otimes_{\Lambda} \Z_p \twoheadrightarrow \bT_{\Gamma_0(n)}^{\circ}.
\end{equation}

Arguments of \cite[\S3]{HPWE}, based on the Riemann--Hurwitz formula, show the following, which includes the ``largeness'' claim regarding $\bT_{\Gamma_1(N)}^\circ$ that is the foundation of the argument in \S\ref{subsec: intro tame family}. 

\begin{lem}[{\cite[\S3.2]{HPWE}}] 
\label{lem: freeness}  
Under the running assumptions (\S \ref{subsec: conventions}), we have 
\begin{enumerate}[label=(\roman*)]
    \item\label{part: freeness lemma rank} The Hecke algebras $\bT_{\Gamma_1(N)}$, $\bT_{\Gamma_1(\ell_0)}$, and $\bT_{\Gamma_1(N)}^{\circ}$ are free $\Lambda$-modules of ranks 
    \[ 
        \rk_{\Lambda}\bT_{\Gamma_1(\ell_0)}=\rk _{\Lambda}\bT_{\Gamma_1(N)}^{\circ}=2, \qquad \rk_{\Lambda}\bT_{\Gamma_1(N)}=4.
    \]
    \item\label{part: freeness lemma isos}  For $n\in \{\ell_0, N\}$, the maps in \eqref{eq: augmentation quotient} are isomorphisms. 

    \item\label{part: freeness lemma specialization} There is an isomorphism $\bT_{\Gamma_1(\ell_0)}^\circ \cong \bT_{\Gamma_0(\ell_0)}^\circ \cong\Lambda/z\Lambda \cong \Z_p$. In particular, $\bT_{\Gamma_1(\ell_0)}^\circ$ is not $\Lambda$-free. 
\end{enumerate}    
\end{lem}
\begin{proof}
Claim \ref{part: freeness lemma isos}, claim \ref{part: freeness lemma specialization}, and the $\Lambda$-freeness claims in \ref{part: freeness lemma rank} may be read directly off from \cite[Thm.\ 3.2.2 and Cor.\ 3.2.3]{HPWE}. The cuspidal cases are the most subtle---see below for a discussion on the failure of $\Lambda$-freeness of $\bT_{\Gamma_1(\ell_0)}^\circ$. The $\Lambda$-ranks claimed in \ref{part: freeness lemma rank} follow from the ranks given in Remark \ref{rem: Hecke w vs U}. 
\end{proof}

More generally, as explained in \cite[\S1]{HPWE}, the failure of $\Lambda$-freeness of the cuspidal Hecke algebra of level $\Gamma_1(n)$ is concentrated at the eigensystem characterized by 
\begin{equation} 
\label{eq: degree}
T_m \mapsto \sigma(m) \text{ for } (m,n) = 1, \quad  \dia{d} \mapsto  1, \quad U_\ell \mapsto \ell \text{ for primes }\ell \mid n. 
\end{equation}
Under the assumptions of \S\ref{subsec: conventions}, this coincides with the mod-$p$ eigensystem defining $\bT_{\Gamma_1(\ell_0)}^{\circ}$. However, the obstruction is resolved when the prime $\ell_1\not\equiv 1\pmod{p}$ is added to the level and we localize at $U_{\ell_1} = 1$, meaning the resulting Hecke ring $\bT_{\Gamma_1(N)}^{\circ}$ is suitably ``large.'' 
The eigensystem in \eqref{eq: degree} is referred to as the \emph{degree eigensystem}, as each Hecke operator is sent to its degree as a correspondence. It formally behaves like the eigensystem corresponding to a stabilization of the weight $2$ Eisenstein series of level $\SL_2(\Z)$, if such a holomorphic modular form existed. 

This phenomenon has a Galois counterpart, which poses a technical difficulty in the formulation of a suitable deformation theory for forms of level $\Gamma_1(n)$; this formulation is one of the main contributions of \cite{HPWE}.
Indeed, in \cite{HPWE}, the authors define certain deformation rings that encode properties of the  $p$-adic Galois representations corresponding to modular forms of level $\Gamma_j(n)$ for $p \nmid n$ squarefree and $j=0,1$. More precisely, these deformation rings parameterize pairs $(D, (u_{\ell})_{\ell \mid n})$, where $D$ is a pseudorepresentation with suitable local conditions and $u_{\ell}$ is a scalar. When the pseudorepresentation comes from a representation, this datum corresponds to the choice of $\Fr_{\ell}$-eigenvalue for a quotient that is unramified at $\ell$. Importantly, for a representation corresponding to a form that is old at $\ell$, such a choice is not unique in general. 
We denote by $R_{\Gamma_j(n)}$ the deformations rings arising from this construction when we restrict to lifts of the residual choice $(\Db = \psi(\omega \oplus 1) : G_{\Q,\cS} \to \F_p, (u_{\ell}=1)_{\ell \mid n})$; this corresponds to restricting to pseudorepresentations arising from modular forms in the  congruence class defined by \eqref{eq: eigensystem}.\footnote{In the notation of \cite{HPWE}, this corresponds to the choices $\Delta=(\Z/n\Z)^{\times}$ and $\Delta=\{1\}$ for $j=0$ and $j=1$, respectively.} 

The deformation ring $R_{\Gamma_1(n)}$ is equipped with a $\Lambda$-module structure induced by the determinant of the Galois representation. Because Galois representation for forms of level $\Gamma_0(n)$  are characterized by the property that the determinant is fixed, it follows that $ R_{\Gamma_1(n)}\otimes _{\Lambda}\Z_p \cong R_{\Gamma_0(n)}$ essentially by definition. Forgetting the datum of $(u_{\ell})_{\ell \mid n}$ gives rise to a morphism $R\to R_{\Gamma_j(n)}$, which can fail to be surjective. This is the case in the settings of interest for $n \in \{\ell_0, N\}$. For example, taking $\kappa: G_{\Q, S} \to \Z_p^{\times}$ to be the $p$-adic cyclotomic character, for $n={\ell_0}$, 
the choices 
\begin{equation} 
\label{eq: eigenvalue choice}
(D = \psi(\kappa \oplus 1), u_{\ell_0}=1)\quad  \text{and} \quad 
(D = \psi(\kappa \oplus 1), u_{\ell_0}=\ell_0)
\end{equation} 
are both valid and belong to the same congruence class as $\ell_0\equiv 1 \pmod p$.

Similar to the mechanism explained in \S \ref{subsec: attach Gal reps to eigenforms}, there are natural maps $\Phi_{\Gamma_j(n)}:  R_{\Gamma_j(n)} \to \bT_{\Gamma_j(n)}$, induced by the pseudorepresentations associated to Hecke eigenforms, characterized by $\Phi_{\Gamma_j(n)}(\Tr_D(\Fr_q))=T_q$ for primes $q \nmid n$ and $\Phi_{\Gamma_j(n)}(u_\ell)=U_\ell$ for primes $\ell \mid n$. These maps are surjective by construction; indeed, $R_{\Gamma_j(N)}$ has an element mapping to every Hecke operator generating $\bT_{\Gamma_1(N)}$ (the operator $T_p$ is a special case; it may be handled exactly as in \cite[Prop.\ 4.1.1]{WWE5}). A simple observation prevents $\Phi_{\Gamma_j(n)}$ from being injective in full generality: the second option in equation \eqref{eq: eigenvalue choice} would correspond to the degree eigensystem \eqref{eq: degree}, which, as previously mentioned, is not realized in the space of modular forms.  

The main conjecture of \cite{HPWE}, specialized in the cases of interest for this paper, states the following: 
\begin{conj} 
\label{conj}  
Under the running assumptions (\S\ref{subsec: conventions}), for $j=0,1$, we have
\begin{enumerate}[label=(\roman*)]
 \item \label{part: conj i} the map $\Phi_{\Gamma_j(N)} : R_{\Gamma_j(N)} \rsurj \bT_{\Gamma_j(N)}$ is an isomorphism 
 \item \label{part: conj ii} the kernel of the map $\Phi_{\Gamma_j(\ell_0)}: R_{\Gamma_j(\ell_0)} \rsurj \bT_{\Gamma_j(\ell_0)}$  isomorphic to $\Lambda/z\Lambda \cong \Z_p$ as a $\Lambda$-module.
\end{enumerate}
\end{conj}

Note that, in particular, part \ref{part: conj ii} implies that the kernel of the map in the statement is principal, which means that the case $u_{\ell_0} = \ell_0$ of \eqref{eq: degree} accounts for the whole kernel. 

The underlying philosophy is that, while an obstruction to $R=\bT$ can occur, it is as small as possible and can be attributed entirely to the degree eigensystem. The authors of \cite{HPWE} aim to prove Conjecture \ref{conj} in follow-up work \cite{APAW}.    

\subsection{A quotient of $\bT^{\circ}_{\Gamma_1(N)}$ is isomorphic to $S$}
\label{subsec: cuspidal quotient} 

We now proceed with step 2 in the strategy outlined at the start of \S\ref{sec: Hecke algs}. Namely, we prove that Conjecture \ref{conj} implies that $\bT_{\Gamma_1(N)}$ has the desired form that is a hypothesis of the main Theorem \ref{thm: main}. We begin by producing the following diagram of $R$-algebras, which combines the $R$-to-$\bT$ comparison of \S\ref{subsec: R=T conj} with a change in level from $N$ to $\ell_0$, that is, a passage to $\ell_1$-old quotients on the Galois and Hecke side. The key to this change in level is the existence of pair of square diagrams, with a reduction modulo $z\Lambda$ map from one to the other, 
\begin{equation}
    \label{eq: level change diagrams}
    \xymatrix@R=0.4em{
    R_{\Gamma_1(N)}\ar@{->>}[r]^(0.45){\alpha_1} \ar[dd]^{\wr}_{\Phi_{\Gamma_1(N)}} & R_{\Gamma_1(\ell_0)}\ar@{->>}[dd]_{\Phi_{\Gamma_1(\ell_0)}} & & & R_{\Gamma_0(N)} \ar[dd]_\wr^{\Phi_{\Gamma_0(N)}} \ar@{->>}[r]^(0.45){\alpha_0} & R_{\Gamma_0(\ell_0)} \ar@{->>}[dd]^{\Phi_{\Gamma_0(\ell_0)}} \\
    & & \ar@{->>}[r]^{\pmod{z\Lambda}\ \ } & & & \\
    \bT_{\Gamma_1(N)}\ar@{->>}[r] & \bT_{\Gamma_1(\ell_0)} & & & \bT_{\Gamma_0(N)} \ar@{->>}[r] & \bT_{\Gamma_0(\ell_0)}
    }
\end{equation}
The existences of $\Phi_{\Gamma_j(n)}$ maps and of the reduction mod $z\Lambda$ map are discussed in \S\ref{subsec: R=T conj}. The rightward map within each square is what we must justify: these exist due to the fact that each $R_{\Gamma_j(n)}$ is a quotient of $R[u_{\ell_0}]$. This is non-trivial because $R_{\Gamma_j(n)}$ is by definition a quotient of $R[u_\ell : \ell \mid n]$, and thus our claim is that $u_{\ell_1} \in R_{\Gamma_j(N)}$ lies in the image of  $R[u_{\ell_0}] \to R_{\Gamma_j(N)}$. Indeed, this follows from virtually the same proof as that of Lemma \ref{lem: natural iso}: $u_{\ell_1}$ is a root of a polynomial in $R[X]$ that reduces mod $\m_R$ to $X^2 - (\ell_1+1)X + \ell_1$, which has distinct roots because $\ell_1 \not\equiv 1 \pmod{p}$ (\S\ref{subsec: conventions}). 

\begin{thm}
    \label{thm: conjs imply result}
    In addition to the running assumptions (\S\ref{subsec: conventions}) and Conjecture \ref{conj}, assume the following:  
    \begin{enumerate}[label=(\roman*)]
        \item\label{part: thm 2} There is an isomorphism $\bT_{\Gamma_1(\ell_0)} \otimes_{\bT_{\Gamma_1(N)}} \bT^\circ_{\Gamma_1(N)} \cong \bT^\circ_{\Gamma_1(\ell_0)}.$
        \item\label{part: thm 3} 
        The element $(u_{\ell_0}-1) \in R_{\Gamma_0(\ell_0)}$ generates the kernel of $\Phi_{\Gamma_0(\ell_0)} : R_{\Gamma_0(\ell_0)} \rsurj \bT_{\Gamma_0(\ell_0)}$.
    \end{enumerate}
    Then there is a surjective map $\pi' :\bT^\circ_{\Gamma_1(N)} \rsurj S$ such that $\pi'(J^\red \cdot \bT^\circ_{\Gamma_1(N)}) = yS$ and such that $\phi_S \circ \pi'$ factors through $\bT^\circ_{\Gamma_1(N)} \rsurj \bar\bT^\circ_{\Gamma_0(N)}$.
\end{thm}

By Theorem \ref{thm: main}, the conclusion of this theorem implies the arithmetic conclusion of \eqref{eq: a1 at ell1 CFT}. 

\begin{rem}
We provide some comments on the assumptions. Assumption \ref{part: thm 2} is a variant at level $\Gamma_1(N)$ of what we prove at level $\Gamma_0(N)$ in \S\ref{sec: Hecke algs} to establish the isomorphisms of \eqref{eq: Hecke transverse}. Assumption \ref{part: thm 3} is rather straightforward to prove but involves details of the theory of $R_{\Gamma_j(n)}$ set up in \cite[\S4]{HPWE}, so we defer it to future work \cite{APAW}. Here, we point out that $u_{\ell_0}-1$ vanishes in $\bT_{\Gamma_0(\ell_0)}$ but does not vanish along the map $R_{\Gamma_0(\ell_0)} \to\Z_p$ corresponding to the right pair in \eqref{eq: eigenvalue choice}.     
\end{rem}

The central idea of the proof of Theorem \ref{thm: conjs imply result} is that the maximal quotient $\bar  \bT_{\Gamma_1(N)}^{\circ,\red}$ of $\bar \bT_{\Gamma_1(N)}^\circ$ that supports a reducible pseudorepresentation has $\F_p$-dimension $2$. The role of this dimension and the fact that it implies the existence of the desired map $\pi'$ has been explained in \S\ref{subsec: intro tame family}.

    This is an interesting phenomenon because even though the Galois pseudorepresentations supported by $\bT_{\Gamma_1(N)}^\circ$ are irreducible and merely residually reducible, there is still (modulo $p$) a non-trivial deformation of reducible Galois representations realized, which more apparently comes from tame families of Eisenstein series. This is comparable to the phenomenon called \textit{extra reducibility} in \cite{wake2020eisenstein}, which occurs in $\bT_{\Gamma_0(\ell_0)}$.

\begin{proof}
This proof focuses on ideals of $R_{\Gamma_1(N)}$ and its quotients, so we set up some notation for these. We recall the notion of reducibility ideal and unramified-at-$\ell_1$ ideal of $R$ set up to state Proposition \ref{prop:B and C}, and denote these by $J^\red$ and $J_1$, respectively. Because every deformation ring and Hecke algebra in this proof carries a pseudorepresentation of $G_{\Q,\cS}$ and thereby has the structure of an $R$-algebra, $J^\red$ and $J_1$ can be extended to any of the rings in this proof.

The key conclusion of Proposition \ref{prop:B and C} is that $J_1 \subset R$ is principal and that $J_1 \bar R = J^\red \bar R =: \bar J^\red$. Thus, the quotient map $\bar  \bT_{\Gamma_1(N)}^\circ \rsurj \bar  \bT_{\Gamma_1(N)}^{\circ,\red}$ has a principal kernel. It will be convenient for the proof to choose a generator $y \in J_1$ (which projects to $\bar R$ to a generator of $\bar J^\red$), so this quotient will also be written as $\bT_{\Gamma_1(N)}^\circ/(y,p)\bT_{\Gamma_1(N)}^\circ$. Because being unramified at $\ell_1$ cuts out $R_{\Gamma_1(\ell_0)}$ as a quotient of $R_{\Gamma_1(N)}$, there is an equality of ideals $J_1R_{\Gamma_1(N)} = \ker \alpha_1$, where $\alpha_1$ is as in \eqref{eq: level change diagrams}.  
For the same reason, $J_1R_{\Gamma_0(N)}=\ker \alpha_0$.

We claim that the image of $y$ in $\bT_{\Gamma_1(N)}^\circ$ makes $\{1,y\}$ a $\Lambda$-basis for it. This follows from applying Nakayama's lemma together with Proposition \ref{prop: cusp Gamma0(N) irreducible}, which says that the reducibility ideal of $\bar J^\red \subset \bar R$ surjects onto the unique non-trivial ideal of $\bT_{\Gamma_1(N)}^\circ \otimes_\Lambda \Lambda/(z,p) \cong \bar\bT_{\Gamma_0(N)}^\circ \simeq \F_p[\ep_1]$. In particular, this justifies the presentation of $\bT_{\Gamma_1(N)}^\circ$ in \S\ref{subsec: intro tame family} as \eqref{eq: T1Ncirc form}.

Our goal is to prove that $\bar \bT_{\Gamma_1(N)}^{\circ,\red} \cong \bT_{\Gamma_1(N)}^\circ/(J^\red,p)\bT_{\Gamma_1(N)}^\circ = \bT_{\Gamma_1(N)}^\circ/(y,p)\bT_{\Gamma_1(N)}^\circ$ has $\F_p$-dimension $2$. To do this, we first study the ideal $y \bT_{\Gamma_1(N)} \subset \bT_{\Gamma_1(N)}$, and then characterize its projection to $\bT_{\Gamma_1(N)}^\circ/p\bT_{\Gamma_1(N)}^\circ$.

Identifying $R_{\Gamma_1(N)} \cong \bT_{\Gamma_1(N)}$ by Conjecture \ref{conj}\ref{part: conj i}, the kernel of the surjection $\Phi_{\Gamma_1(\ell_0)} \circ \alpha_1 :\bT_{\Gamma_1(N)} \rsurj \bT_{\Gamma_1(\ell_0)}$ is $\Lambda$-free of rank $2$ by Lemma \ref{lem: freeness}, which states that the source and target of the surjection are $\Lambda$-free of ranks $4$ and $2$, respectively. Likewise, this kernel is compatible with reduction modulo $z\Lambda$, in the sense that the natural map
\begin{equation}
\label{eq: kernel surj}
    \ker (\Phi_{\Gamma_1(\ell_0)} \circ \alpha_1) \otimes_\Lambda \Z_p 
    \to \ker \left(\Phi_{\Gamma_0(\ell_0)} \circ \alpha_0 : \bT_{\Gamma_0(N)} \rsurj \bT_{\Gamma_0(\ell_0)}\right)    
\end{equation}
is an isomorphism of free $\Z_p$-modules of rank $2$. 

Next, we establish the existence of an element $g \in \ker (\Phi_{\Gamma_1(\ell_0)} \circ \alpha_1) \subset \bT_{\Gamma_1(N)}$ satisfying the following properties. 
\begin{enumerate}[label=(\alph*)]
    \item\label{part: def g generates kernel} The image of $g$ under $\alpha_1 : R_{\Gamma_1(N)} \cong \bT_{\Gamma_1(N)} \rsurj R_{\Gamma_1(\ell_0)}$ generates $\ker \Phi_{\Gamma_1(\ell_0)}$. 
    \item\label{part: def g vanishes} The image of $g$ under $\bT_{\Gamma_1(N)} \rsurj \bT_{\Gamma_0(N)}^\circ$ vanishes.
    \item\label{part: def g g'} The image $g'$ of $g$ under $\bT_{\Gamma_1(N)} \rsurj R_{\Gamma_0(N)} \cong \bT_{\Gamma_0(N)}$ makes $\{y,g'\} \subset R_{\Gamma_0(N)}$ a $\Z_p$-basis for $\ker (\Phi_{\Gamma_0(\ell_0)} \circ \alpha_0)$.
    \item\label{part: def g lambda basis} The subset $\{y,g\} \subset \ker (\Phi_{\Gamma_1(\ell_0)} \circ \alpha_1)$ a $\Lambda$-basis.
\end{enumerate}
Indeed, consider the surjection $\beta : \ker (\Phi_{\Gamma_0(\ell_0)} \circ \alpha_0) \rsurj \ker \Phi_{\Gamma_0(\ell_0)}$ given by restriction of $\alpha_0$; the source has $\Z_p$-rank $2$ and the target has $\Z_p$-rank $1$. Now $\ker \beta = J_1 R_{\Gamma_0(N)}$ is $\Z_p$-free of rank $1$ and spanned by the image of $y$ in $R_{\Gamma_0(N)}$. Let 
\[
g' = (u_{\ell_0}-1) \in \ker (\Phi_{\Gamma_0(\ell_0)} \circ \alpha_0) \subset R_{\Gamma_0(N)},
\] 
which is a lift over $\beta$ of the generator (by assumption \ref{part: thm 3}) $u_{\ell_0}-1$ of $\ker \Phi_{\Gamma_0(\ell_0)}$. Thus, $\{y,g'\}$ is a $\Z_p$-basis of $\ker (\Phi_{\Gamma_0(\ell_0)} \circ \alpha_0)$. Now let $g$ denote a lift of $g'$ over the reduction mod $z\Lambda$ map $\ker (\Phi_{\Gamma_1(\ell_0)} \circ \alpha_1) \rsurj \ker (\Phi_{\Gamma_0(\ell_0)} \circ \alpha_0)$ (see \eqref{eq: kernel surj}). 

We claim that our choice of $g$ satisfies \ref{part: def g generates kernel}--\ref{part: def g lambda basis}. Property \ref{part: def g generates kernel} follows from the fact that the projection of $g$ along $\bT_{\Gamma_1(N)} \rsurj R_{\Gamma_0(\ell_0)}$ generates $\ker \Phi_{\Gamma_0(\ell_0)}$ by choice of $g'$, along with the consequence of Conjecture \ref{conj}\ref{part: conj ii} that the surjection $R_{\Gamma_1(\ell_0)} \rsurj R_{\Gamma_0(\ell_0)}$ induces an isomorphism $\ker \Phi_{\Gamma_1(\ell_0)} \isoto \ker \Phi_{\Gamma_0(\ell_0)}$. Property \ref{part: def g vanishes} follows from the facts that $g$ maps to $g'=u_{\ell_0} -1 \in R_{\Gamma_0(N)}$, that the surjection $\bT_{\Gamma_1(N)} \rsurj \bT_{\Gamma_0(N)}^\circ$ factors through $\bT_{\Gamma_1(N)} \rsurj R_{\Gamma_0(N)}$, and that $U_{\ell_0}$ acts as $1$ on the two $\Z_p$-valued cusp forms whose eigensystems are parameterized by $\bT_{\Gamma_0(N)}^\circ$. Property \ref{part: def g g'} was already mentioned above, after choosing $g'$. Property \ref{part: def g lambda basis} follows from \ref{part: def g g'} by Nakayama's lemma. 

Applying Conjecture \ref{conj}\ref{part: conj ii} and \ref{part: def g generates kernel}, the annihilator ideal in $\Lambda$ of $\ker \Phi_{\Gamma_1(\ell_0)} = g R_{\Gamma_1(\ell_0)}$ is $z\Lambda$ (thus, for $x \in \Lambda$, $xg \in y\bT_{\Gamma_1(N)}$ if and only if $x \in z\Lambda$). Consequently, applying \ref{part: def g lambda basis}, we must have a $\Lambda$-linear presentation for $y\bT_{\Gamma_1(N)}$ within the $\Lambda$-free module $(y,g)\bT_{\Gamma_1(N)} = \ker (\Phi_{\Gamma_1(\ell_0)} \circ \alpha_1)$ with $\Lambda$-basis $\{y,g\}$ as 
\[
\Lambda \cdot y \oplus \Lambda \cdot g \cong (y,g)\bT_{\Gamma_1(N)} \supset y\bT_{\Gamma_1(N)} \cong \Lambda \cdot y \oplus (z\Lambda) \cdot g. 
\] 

Now, by assumption \ref{part: thm 2} and the fact that $\bT_{\Gamma_1(N)}/(y,g)\bT_{\Gamma_1(N)} \cong \bT_{\Gamma_1(\ell_0)}$, we have (by Lemma \ref{lem: freeness}\ref{part: freeness lemma specialization}) that
\[
    \bT_{\Gamma_1(N)}^\circ/(y,g)\bT_{\Gamma_1(N)}^\circ \cong \bT_{\Gamma_1(\ell_0)}^\circ \cong \Z_p, \qquad \text{thus } \bT_{\Gamma_1(N)}^\circ/(y,g,p)\bT_{\Gamma_1(N)}^\circ \cong \F_p. 
\]
We see that $(y,g)\bar \bT_{\Gamma_1(N)}^\circ \subset \bar \bT_{\Gamma_1(N)}^\circ$ is the maximal ideal, as is $(y,z)\bar \bT_{\Gamma_1(N)}^\circ$. Combining this with the fact that $zg \in y \bar \bT_{\Gamma_1(N)}^\circ$, we deduce that $\dim_{\F_p} \bT_{\Gamma_1(N)}^\circ/(y,p)\bT_{\Gamma_1(N)}^\circ \leq 2$. 

On the other hand, the ideal $y\bar\bT_{\Gamma_1(N)}^\circ$ is the $\Lambda/p\Lambda$-span of the images of $y$ and $zg$. Now we apply property \ref{part: def g g'} of $g$, which is that it projects to $\bT_{\Gamma_1(N)}^\circ/z\bT_{\Gamma_1(N)}^\circ \cong \bT_{\Gamma_0(N)}^\circ$ as zero. Thus, the image of $g$ in $\bar\bT_{\Gamma_1(N)}^\circ$ lies in $z\bar\bT_{\Gamma_1(N)}^\circ$, from which it follows that $y\bar\bT_{\Gamma_1(N)}^\circ$ is the $\Lambda/p\Lambda$-span of two elements, one of which lies in $z^2(\Lambda/p\Lambda)$. We conclude that $\dim_{\F_p} \bar\bT_{\Gamma_1(N)}^\circ/y\bar\bT_{\Gamma_1(N)}^\circ = 2$. Then the argument in \S\ref{subsec: intro tame family} completes the proof.
\end{proof}

\subsection*{Acknowledgements}

The authors would like to thank Barry Mazur and Ken Ribet for their inspiring work on the topics discussed in this article. CH was supported by an AMS-Simons Research Enhancement Grant for Primarily Undergraduate Institution Faculty. PW was supported by a grant from the National Science Foundation DMS-2337830 and by a fellowship from Simons Foundation International [SFI-MPS-SFM-00021789, PW]. CWE was supported by a grant from the Simons Foundation MPS-TSM-00846912 and a grant from the National Science Foundation DMS-2401384. 

The authors used an AI tool (Claude Opus 5.0) to simplify the ring presentations of \S\ref{subsec: computation} from SageMath output and subsequently independently verified the presentation. This was the only use of AI tools. 

\bibliographystyle{alpha}

\bibliography{metabelian.bib}

\newcommand{\etalchar}[1]{$^{#1}$}
\def\cprime{$'$} \def\Dbar{\leavevmode\lower.6ex\hbox to 0pt{\hskip-.23ex
  \accent"16\hss}D} \def\cfac#1{\ifmmode\setbox7\hbox{$\accent"5E#1$}\else
  \setbox7\hbox{\accent"5E#1}\penalty 10000\relax\fi\raise 1\ht7
  \hbox{\lower1.15ex\hbox to 1\wd7{\hss\accent"13\hss}}\penalty 10000
  \hskip-1\wd7\penalty 10000\box7}
  \def\cftil#1{\ifmmode\setbox7\hbox{$\accent"5E#1$}\else
  \setbox7\hbox{\accent"5E#1}\penalty 10000\relax\fi\raise 1\ht7
  \hbox{\lower1.15ex\hbox to 1\wd7{\hss\accent"7E\hss}}\penalty 10000
  \hskip-1\wd7\penalty 10000\box7} \def\Dbar{\leavevmode\lower.6ex\hbox to
  0pt{\hskip-.23ex \accent"16\hss}D}
  \def\cfac#1{\ifmmode\setbox7\hbox{$\accent"5E#1$}\else
  \setbox7\hbox{\accent"5E#1}\penalty 10000\relax\fi\raise 1\ht7
  \hbox{\lower1.15ex\hbox to 1\wd7{\hss\accent"13\hss}}\penalty 10000
  \hskip-1\wd7\penalty 10000\box7}
  \def\cftil#1{\ifmmode\setbox7\hbox{$\accent"5E#1$}\else
  \setbox7\hbox{\accent"5E#1}\penalty 10000\relax\fi\raise 1\ht7
  \hbox{\lower1.15ex\hbox to 1\wd7{\hss\accent"7E\hss}}\penalty 10000
  \hskip-1\wd7\penalty 10000\box7}
\begin{thebibliography}{BHRS{\etalchar{+}}26}

\bibitem[BC09]{BC2009}
Jo{\"e}l Bella{\"{\i}}che and Ga{\"e}tan Chenevier.
\newblock Families of {G}alois representations and {S}elmer groups.
\newblock {\em Asterisque}, 324:xii+314 pp., 2009.

\bibitem[BHRS{\etalchar{+}}26]{APAW}
Amie Bray, Catherine Hsu, \'{O}scar Rivero~Salgado, Alice Pozzi, Vinayak
  Vatsal, and Carl Wang-Erickson.
\newblock {Universal modularity conjectures for elliptic modular forms:
  {M}odularity lifting}.
\newblock In preparation, 2026.

\bibitem[CE05]{CE2005}
Frank Calegari and Matthew Emerton.
\newblock On the ramification of {H}ecke algebras at {E}isenstein primes.
\newblock {\em Invent. Math.}, 160(1):97--144, 2005.

\bibitem[Che14]{chen2014}
Ga\"{e}tan Chenevier.
\newblock The $p$-adic analytic space of pseudocharacters of a profinite group,
  and pseudorepresentations over arbitrary rings.
\newblock In {\em Automorphic Forms and {G}alois Representations: {V}ol.\ {I}},
  volume 414 of {\em London Mathematical Society Lecture Note Series}, pages
  221--285. Cambridge Univ. Press, Cambridge, 2014.
\newblock We follow the numbering of the online version
  \url{https://arxiv.org/abs/0809.0415v2}, which differs from the print
  version.

\bibitem[DDT94]{DDT1994}
Henri Darmon, Fred Diamond, and Richard Taylor.
\newblock Fermat's last theorem.
\newblock In {\em Current developments in mathematics, 1995 ({C}ambridge,
  {MA})}, pages 1--154. Int. Press, Cambridge, MA, 1994.

\bibitem[HPWE26]{HPWE}
Catherine Hsu, Alice Pozzi, and Carl Wang-Erickson.
\newblock Universal modularity conjectures for elliptic modular forms:
  {O}bstructions.
\newblock Preprint, 2026.

\bibitem[HWWE23]{part2}
Catherine Hsu, Preston Wake, and Carl Wang-Erickson.
\newblock Explicit non-{G}orenstein {$R={\mathbb{T}}$} via rank bounds {II}:
  {C}omputational aspects.
\newblock {\em Res. Number Theory}, 9(1):\ Paper No. 16, 52 pp., 2023.
\newblock Part of the proceedings of the Fifteenth Algorithmic Number Theory
  Symposium (ANTS-XV).

\bibitem[HWWE25]{part1}
Catherine Hsu, Preston Wake, and Carl Wang-Erickson.
\newblock Explicit non-{G}orenstein {$R={\Bbb {T}}$} via rank bounds {I}:
  {D}eformation theory.
\newblock {\em Selecta Math. (N.S.)}, 31(1):Paper No. 18, 71, 2025.

\bibitem[LMP26]{LMP}
Jaclyn Lang, Katharina M\"uller, and Bharathwaj Palvannan.
\newblock {Modularity theorems for Eisenstein congruences in prime-power
  level}.
\newblock arXiv:2608.20125, 2026.

\bibitem[LPW26]{LPW}
Jaclyn Lang, Robert Pollack, and Preston Wake.
\newblock {Counting level-raising congruences using modular representation
  theory}.
\newblock arXiv:2609.05144, 2026.

\bibitem[Maz77]{mazur1978}
B.~Mazur.
\newblock Modular curves and the {E}isenstein ideal.
\newblock {\em Inst. Hautes \'Etudes Sci. Publ. Math.}, (47):33--186 (1978),
  1977.

\bibitem[Maz89]{mazur1989}
B.~Mazur.
\newblock Deforming {G}alois representations.
\newblock In {\em Galois groups over {${\bf Q}$} ({B}erkeley, {CA}, 1987)},
  volume~16 of {\em Math. Sci. Res. Inst. Publ.}, pages 385--437. Springer, New
  York, 1989.

\bibitem[Oht14]{ohta2014}
Masami Ohta.
\newblock Eisenstein ideals and the rational torsion subgroups of modular
  {J}acobian varieties {II}.
\newblock {\em Tokyo J. Math.}, 37(2):273--318, 2014.

\bibitem[Rib76]{ribet1976}
Kenneth~A. Ribet.
\newblock A modular construction of unramified {$p$}-extensions of
  {$\mathbf{Q}(\mu _{p})$}.
\newblock {\em Invent. Math.}, 34(3):151--162, 1976.

\bibitem[Sea18]{SAGE}
William~A Stein and et~al.
\newblock {\em {S}ageMath, the {S}age {M}athematics {S}oftware {S}ystem}.
\newblock The Sage Development Team, 2018.
\newblock {\tt http://www.sagemath.org}.

\bibitem[SW97]{SW1997}
C.~M. Skinner and A.~J. Wiles.
\newblock Ordinary representations and modular forms.
\newblock {\em Proc. Nat. Acad. Sci. U.S.A.}, 94(20):10520--10527, 1997.

\bibitem[Wak23]{wake2020eisenstein}
Preston Wake.
\newblock The {E}isenstein ideal for weight {$k$} and a {B}loch-{K}ato
  conjecture for tame families.
\newblock {\em J. Eur. Math. Soc. (JEMS)}, 25(7):2815--2861, 2023.

\bibitem[WE20]{CarlA-inf}
Carl Wang-Erickson.
\newblock Presentations of non-commutative deformation rings via
  ${A}_\infty$-algebras and applications to deformations of {G}alois
  representations and pseudorepresentations.
\newblock arXiv:1809.02484v2 [math.NT], 2020.

\bibitem[WWE19]{WWE4}
Preston Wake and Carl Wang-Erickson.
\newblock Deformation conditions for pseudorepresentations.
\newblock {\em Forum Math. Sigma}, 7:e20, 2019.

\bibitem[WWE20]{WWE3}
Preston Wake and Carl Wang-Erickson.
\newblock The rank of {M}azur's {E}isenstein ideal.
\newblock {\em Duke Math. J.}, 169(1):31--115, 2020.

\bibitem[WWE21]{WWE5}
Preston Wake and Carl Wang-Erickson.
\newblock The {E}isenstein ideal with squarefree level.
\newblock {\em Adv. Math.}, 380:Paper No. 107543, 62, 2021.

\bibitem[Yoo19]{yoo2019}
Hwajong Yoo.
\newblock Non-optimal levels of a reducible {$\rm{mod}\,\ell$} modular
  representation.
\newblock {\em Trans. Amer. Math. Soc.}, 371(6):3805--3830, 2019.

\end{thebibliography}

\end{document}